\documentclass[hidelinks,onefignum,onetabnum]{siamart251216}

\pdfoutput=1

\usepackage{lipsum}
\usepackage{amsfonts}
\usepackage{graphicx}
\usepackage{epstopdf}
\usepackage{algorithmic}
\ifpdf
  \DeclareGraphicsExtensions{.eps,.pdf,.png,.jpg}
\else
  \DeclareGraphicsExtensions{.eps}
\fi

\usepackage{amssymb}
\usepackage{amsmath}
\allowdisplaybreaks
\usepackage{cancel}
\usepackage{graphicx} 
\usepackage{animate} 
\usepackage{subcaption}
\usepackage[export]{adjustbox}
\usepackage{xcolor}

\newsiamremark{remark}{Remark}
\newsiamremark{hypothesis}{Hypothesis}
\crefname{hypothesis}{Hypothesis}{Hypotheses}
\newsiamthm{claim}{Claim}
\newsiamremark{example}{Example}
\newsiamremark{examples}{Examples}
\newsiamremark{case}{Case}
\newsiamthm{assumption}{Assumption}
\newsiamremark{fact}{Fact}
\crefname{fact}{Fact}{Facts}

\headers{Neural SHP with Boundary Conditions on Graphs}{Yaxin Feng, Yang Xiang, and Haomin, Zhou}

\title{A Parameterized Stochastic Hamiltonian Process on Graphs with Two-Point Boundary Conditions
}

\author{Yaxin Feng\thanks{Department of Mathematics, The Hong Kong University of Science and Technology,
Hong Kong SAR, China (\email{yfengba@connect.ust.hk}, \email{maxiang@ust.hk}).}
\and Yang Xiang\footnotemark[1]\textsuperscript{ }\thanks{HKUST Shenzhen-Hong Kong Collaborative Innovation Research Institute, Futian, Shenzhen, China.}
\and Haomin Zhou\thanks{School of Mathematics, Georgia Institute of Technology, GA 30332 USA (\email{hmzhou@math.gatech.edu}).}
}

\usepackage{amsopn}

\usepackage{booktabs}

\ifpdf
\hypersetup{
 pdftitle={Neural Parameterized Mean-Field Controls on Graphs: A Stochastic Hamiltonian Process with Two-Point Boundary Conditions},
 pdfauthor={Yaxin Feng, Yang Xiang, and Haomin, Zhou}
}
\fi

\IfFileExists{ex_supplement.aux}{%
 \externaldocument[][nocite]{ex_supplement}%
}{}

\begin{document}

\maketitle



\begin{abstract}
This paper presents a Lagrangian framework to solve Wasserstein Hamiltonian flows (WHF) under two-point boundary conditions on large finite graphs. Characterized by a coupled system of discrete master equation and Hamilton-Jacobi (HJ) equation on the probability simplex, this formulation defines the macroscopic WHF. While classical density-based solvers can track this flow, they face severe memory and computational bottlenecks as the graph scale grows. To bypass these limitations, we introduce the stochastic Hamiltonian process (SHP), a sample-based mathematical framework that tracks individual particle dynamics rather than the full node-wise density. 
We parameterize the time-dependent potential using neural network, yielding the Neural SHP solver. Numerical experiments on synthetic graphs demonstrate the scalability of our method relative to density-based baselines. Furthermore, we validate Neural SHP across two large-scale real-world networks, a PubMed citation network and a massive road network, demonstrating its robust efficiency for graph partition and routing problems on complex graph structures.

\end{abstract}

\begin{keywords}
stochastic Hamiltonian process, large-scale graphs, Lagrangian formulation, particle methods
\end{keywords}


\section{Introduction}

In this work, we investigate stochastic Hamiltonian process (SHP) on finite graphs within Wasserstein geometry, focusing primarily on the microscopic stochastic process $X(t)$ with density $\rho(t)$. This process governs particle trajectories whose transition rates $Q(t)$ are driven by the graph gradient of a potential field $S(t)$. Macroscopically, the transport of the global probability density can be formally characterized by a deterministic Wasserstein Hamiltonian Flow (WHF)~\cite{chow2020wasserstein}, which is given by a coupled system govening by the density $\rho$ and potential $S$:
\begin{equation}\label{eq:WHs0}
\begin{pmatrix}
\frac{d \rho}{dt}\\[1mm]
\frac{d S}{dt}
\end{pmatrix}
=
\begin{pmatrix}
0&\mathbf{I}_n\\
-\mathbf{I}_n&0
\end{pmatrix}
\begin{pmatrix}
\frac{\partial\mathcal H}{\partial \rho}\\[1mm]
\frac{\partial\mathcal H}{\partial S}
\end{pmatrix},
\end{equation}
where $\mathcal{H}(\rho,S)$ is the Hamiltonian on an $n$-node graph $\mathsf{G}$. Here, $\rho$ and $S$ are vector-valued functions defined on the node set $V$, and $\mathbf{I}_n$ is the $n \times n$ identity matrix.
While directly solving this macroscopic WHF becomes computationally intractable on large-scale graphs due to the massive system of coupled equations, focusing on the particle-level process in the Lagrangian viewpoint bypasses these global bottlenecks while strictly preserving the underlying Hamiltonian structure.

Wasserstein dynamics on finite graphs have been studied extensively for both dissipative and conservative flows. Under this geometric framework, graph-based Fokker-Planck equations can be interpreted as free-energy gradient flows or Kolmogorov forward equations for Markov processes~\cite{chow2012fokker,chow2017entropy}, building upon discrete optimal transport theories established in~\cite{maas2011gradient,mielke2011gradient,mielke2014relation}. Recently, optimizing and controlling transition rates in continuous-time Markov chain (CTMC) has extended to discrete generative and diffusion frameworks~\cite{campbell2023continuous,benton2024denoising}. We move beyond these purely dissipative regimes to WHF which provides a natural geometric formulation to model conservative and coupled dynamics on discrete graphs.

Depending on the specific variational formulation or optimal control problem, the WHF can be configured with distinct initial and boundary conditions. For instance, discrete mean-field game problems on finite states or graphs are constrained by an initial density $\rho(0)$ and a terminal condition $S(T)$~\cite{gomes2010discrete,gomes2013continuous,luo2023deep,feng2026discrete}. Stochastic optimal transport and Schrödinger bridge problems specify fixed initial and terminal density conditions $\rho(0)$ and $\rho(T)$ \cite{chow2022dynamical,theodoropoulos2026generalized}, whereas the discrete Schrödinger equation requires initial conditions on both $\rho(0)$ and $S(0)$ \cite{chow2019discrete,wu2026parameterized}. 

Beyond these classical settings, WHF with two-point boundary conditions are increasingly applied to model complex dynamics directly on static finite graphs, representing the fixed, discrete topology of real-world networks. These real-world network applications include routing and traffic flows \cite{camilli2016stationary,achdou2024deterministic,carlier2008optimal}, the propagation of opinions across social networks \cite{bauso2016opinion,fabian2024learning,yang2018learning}, and evolution of chemical reaction networks \cite{maas2020modeling}. However, most existing studies remain density-based, reflecting an Eulerian formulation that becomes prohibitively expensive in memory and computation on large-scale graphs. This scalability bottleneck heavily motivates our Lagrangian particle approach.

In this work, we propose a parameterized computational framework for SHP with two-point boundary value conditions on a graph. 
This framework is constructed using a Lagrangian formulation of the SHP, together with a neural network parameterization implemented through particle trajectories. It elliminates the scalibility issue on large-scale graph encountered by the density-based methods. 
More precisely, we consider~\cref{eq:WHs0}, in which the Hamiltonian is defined by 
\begin{equation}\label{eq:Hamiltonian}
  \mathcal{H}(\rho,S)=\mathcal K(\rho,S) + \mathcal F(\rho),
\end{equation}  
where $\mathcal{K}(\rho,S)$ is the kinetic energy, $\mathcal{F}(\rho)$ is the potential.
The boundary condition are given by $\rho(0)=\rho_0$ and $S(T)=-\frac{\partial \mathcal{G}}{\partial \rho}(T),$ where $\mathcal{G}(\rho(T))$ is the terminal potential. 
This is a forward-backward structured system.
Solving a coupled forward-backward system is, in general, numerically challenging. We overcome this by minimizing a Markov-process-constrained action integral functional (hereafter, the action functional), which efficiently manages the forward-backward dynamics as detailed in~\cref{sec:2ptBdySHP}.

Here are some examples that can be handled by our proposed SHP framework.

\begin{example}[Mean-Field Games on Graphs]\label{ex:MFG}

A mean-field game on a graph (Graph MFG) describes indistinguishable, noncooperative agents interacting through their population distribution~\cite{gomes2013continuous,cecchin2020probabilistic,feng2026discrete,bayraktar2018analysis,achdou2024deterministic,cecchin2019convergence}. The WHF~\cref{eq:WHs0} with two-point boundary conditions can be used to describe the Nash equilibrium of a mean-field game. When the running cost $f_i$ and terminal cost $g_i$ of an agent at node $i\in V$ are defined as the first variations of the population potentials $\mathcal{F}$ and $\mathcal{G}$,
$$
f_i(\rho)=\frac{\partial \mathcal{F}}{\partial \rho_i},
\qquad
g_i(\rho(T))=\frac{\partial \mathcal{G}}{\partial \rho_i}(T),
\quad i\in V,$$
the system corresponds to a variational mean-field game on graphs, which can be solved by minimizing the constrained action functional~\cref{eq:ActionFunctional}.
\end{example}

\begin{example}[Mean-Field Control on Graphs]\label{ex:MFCs}

Mean-field control on a graph (Graph MFC) describes a cooperative population whose common strategy is chosen by a central planner~\cite{fabian2025learning,dan2026variational}. The WHF~\cref{eq:WHs0} with two-point boundary conditions is the first-order optimality condition of this kind of problem. Suppose that the population-dependent running and terminal costs of an agent at node $i\in V$ are $F_i(\rho)$ and $G_i(\rho(T))$, respectively. The population potentials are
$$
\mathcal{F}(\rho)= \mathbb{E} [F(\rho)] = \sum_{k\in V} F_k(\rho)\rho_k,
\quad
\mathcal{G}(\rho(T))=\mathbb{E}[G(\rho(T))]=\sum_{k\in V} {\rho_k(T)}G_k(\rho(T)).
$$
The central planner minimizes the constrained action functional~\eqref{eq:ActionFunctional} subject to these potentials.
Here we have 
$$
\frac{\partial \mathcal{F}}{\partial \rho_i}
= \sum_{k\in V}  \left( \frac{\partial \rho_k}{\partial \rho_i} F_k(\rho) + \rho_k \frac{\partial F_k(\rho)}{\partial \rho_i} \right)
= F_i(\rho) + \sum_{k\in V} \rho_k \frac{\partial F_k(\rho)}{\partial \rho_i}
$$ and similar to $\mathcal{G}$. 
\end{example}

\begin{remark}[Relation to Optimal Transport on Graphs]
When $\mathcal F\equiv0$, the MFC problem minimizes the sum of
the kinetic energy $\mathcal K$ and the terminal condition $\mathcal G$.
In particular, 
given the initial condition $\rho(0)=\rho_0$, and the terminal condition $\mathcal G(\rho(T))=0\text{ }(\text{when } \rho(T)=\rho_T) \text{ or } +\infty\text{ }(\text{when } \rho(T)\neq\rho_T)$ is taken to enforce the terminal constraint $\rho(T)=\rho_T$ and recover the dynamical optimal transport on graphs.
\end{remark}



The paper is organized as follows. We introduce the time-inhomogeneous Markov process and the SHP on finite graphs in~\cref{sec:GMFG}, and formulate the SHP with two-point boundary conditions alongside the parameterized action functional in~\cref{sec:2ptBdySHP}. Algorithmic details, complexity, and implementation are discussed in~\cref{sec:methodology}, followed by numerical validation in~\cref{sec:experiments}.

\section{Stochastic Hamiltonian Process on Finite Graphs}
\label{sec:GMFG}
In this section, we review the definitions and theorems of~\cite{cui2021stochastic} and tailor them to our problem.

We consider an undirected, connected graph $\mathsf{G}=(V,E,\mathbf{W})$ with no self-loops or multiple edges~\cite{gross2003handbook,feng2026discrete}. Let $V=\{s_i\}_{i=1}^n$ denote the node set, where $n:=\vert{}V\vert{}$ is the number of nodes. The edge between nodes $s_i$ and $s_j$ is denoted by $(i,j) \in E$, with $N(i)=\{s_j \in V : (i,j) \in E\}$ representing the neighbor set of $s_i$ for each $i \in \{1, \dots, n\}$. In addition, $w_{ij} \in \mathbf{W}$ is the symmetric weight of the edge $(i,j) \in E$ such that $w_{ij} = w_{ji} > 0$. The maximum degree of $\mathsf{G}$ is given by $d_\text{max} := \max_{i \in V} \vert{}N(i)\vert{}$.

The probability set (simplex) supported on the finite graph $\mathsf{G}$ is defined as 
$$\mathcal P(\mathsf{G}):=\left\{\rho=(\rho_i)_{i\in V}\in\mathbb R^n\;\middle|\;\rho_i\geq0\ \text{for all }i\in V,
\quad\sum_{i\in V}\rho_i=1\right\}.
$$ where $\rho_i$ is the probability on the node $i$. The interior of $\mathcal{P}(\mathsf{G})$ is denoted by $\mathcal{P}_o(\mathsf{G})$, in which $\rho_i > 0$ on each node.

The vector field $v_{ij}=-v_{ji}$ on graph $\mathsf{G}$ is a $n\times n$ skew-symmetric matrix defined on the edge set $E$.

Given two vector fields $v$ and $\tilde{v}$, the discrete inner product of the vector fields on the finite graph $\mathsf{G}$ and $\rho \in \mathcal{P}(\mathsf{G})$ is $$\langle v,\tilde{v}\rangle_{c(\rho)}=\frac{1}{2}\sum_{(i,j) \in E}v_{ij}\tilde{v}_{ij}c_{ij}(\rho),$$ where $c_{ij}(\rho)$ being the weighted function of $\rho_i$ and $\rho_j$, and the factor $\frac{1}{2}$ is needed because all the values defined on the edges are counted twice. In this work, we apply the upwind density to ensure its positive: $c_{ij}^U(\rho) = \rho_i \mathbb{I}_{\{S_j > S_i\}} + \rho_j \mathbb{I}_{\{S_i > S_j\}}$ where $\mathbb{I}(\cdot)$ is the indicator function. We have $\langle v,v\rangle_{c(\rho)}=\frac{1}{2}\sum_{(i,j) \in E}v_{ij}^2c_{ij}(\rho)$.

Define the divergence of the discrete flux $\rho v$ on $\mathsf{G}$ as
$$\text{div}_{\mathsf{G}}(\rho v) := \left(\sum_{j \in N(i)} \sqrt{w_{ij}}v_{ij} c_{ij}(\rho)\right)_{i=1}^n,$$
where $w_{ij}$'s are the weights on the edges.

The gradient on the graph $\nabla_{\mathsf{G}}$ can be obtained from the dual operation via integration by parts under the inner product defined above: 
$$
-\sum_{i\in V} \text{div}_{\mathsf{G}} (\rho v)|_i \xi_i = \langle \nabla_{\mathsf{G}} \xi, v\rangle_\rho,
$$
in which the discrete scalar function $\xi=(\xi_i)_{i\in V} :V\rightarrow \mathbb{R}$ induces a vector field 
$$(\nabla_{\mathsf{G}}\xi)_{ij}:=\sqrt{w_{ij}}(\xi_j-\xi_i)_{(i,j) \in E}.$$ 

For simplicity, we consider $w_{ij}=w_{ji}=1, i,j \in V$ in this paper.

\begin{remark}
Other choices for $c_{ij}$, such as the simple average $c_{ij}^A(\rho) = \frac{\rho_i+\rho_j}{2}$, the logarithmic weight $c_{ij}^L(\mathbf\rho)=\frac{\rho_i-\rho_j}{\log \left(\rho_i\right)-\log \left(\rho_j\right)}$, \textit{etc.}, have been used in previous studies~\cite{feng2026discrete,chow2019discrete,maas2011gradient}. 
However, weights of this type do not admit a time-inhomogeneous Markov process $X(t)$ due to the possibility of negative transition probabilities~\cite{cui2021stochastic}. Therefore, they are not considered in this study.
\end{remark}\label{rmk:theta}


\subsection{Time Inhomogeneous Markov Process}

Given an initial distribution $\rho_0$, 
let $\{X(t)\}_{t\geq 0}$ be a stochastic process initialized by $X(0)\sim \rho_
0.$
Within $0\leq \tau \leq t,$ the transition kernel $P_{\tau,t}$ defines the state transition of the particles: $${[P_{\tau,t}]}_{ij}=P(X(t)=s_j\mid X(\tau)=s_i).$$ 
$X$ is called time inhomogeneous Markov process if $P$ is time-dependent. 
The transition matrix satisfies the forward Kolmogorov equation $$\frac{\partial}{\partial t} P_{\tau,t} = P_{\tau,t} Q(t), P_{\tau,\tau}=I,$$ where $Q(t)$ (henceforth written as $Q$) is the time-dependent transition rate matrix, or infinitesimal generator.

The off-diagonal entries of $Q$ satisfy $Q_{ij}\geq 0$ for $i \neq j,$ and each row sums to zero $$\sum_{j=1}^n Q_{ij}= 0.$$ Consequently, the diagonal entries are given by $$Q_{ii}=-\sum_{j\neq i} Q_{ij}\leq0.$$
Furthermore, the marginal probability vector $\rho=(\mathbb{P}(X=s_i))_{i=1}^n$ satisfies~\cite{cui2021stochastic,kolokoltsov2010nonlinear} $$\frac{d \rho}{dt}=\rho\,Q, \quad \rho(t) = \rho(\tau) P_{\tau,t}.$$

Letting $t_m=mh,m=0,\cdots,M,$ where $h=\frac{T}{M},$ we apply the forward Euler approximation of the transition rate $Q$ by transition probability matrix $P$. The transition probability matrix is defined by 
\begin{equation}\label{eq:EulerTransition}
 P_m^h:= I+hQ(t_m)
\end{equation} within one step. 
The step size $h$ is chosen such that $1+h[Q(t_m)]_{ii}\geq 0$ for all $i$ and $m$. 
Only the nonzero candidate transitions are stored. Consequently, each row contains at most $d_\text{max}+1$ entries, including the stay-put probability. Thus, the matrix $[P_m^h]$ has dimensions $N \times (d_\text{max}+1)$.

\subsection{Stochastic Hamiltonian Process}\label{subsec:SHP}

The stochastic process $\{X(t)\}_{t\geq 0}$ is a Hamiltonian process on the graph if the below conditions are satisfied. 
 
The density of $X(t)$ satisfies the generalized master equation ofa time inhomogeneous Markov chain, given by $\frac{d \rho}{dt}=\rho Q.$ 

The relationship between the transition-rate matrix $Q$ and the potential function $S$ is:
\begin{equation}\label{eq:parameterized-Q}
 Q_{ij} = \begin{cases} (S_j-S_i)^+, & j \in N(i), \\ - \sum_{k \in N(i)} (S_k-S_i)^+, & j = i, \\ 0, & j \notin N(i) \cup \{i\}. \end{cases}
\end{equation}
Here, $(x)^+=\max(x,0),(x)^-=\min(x,0), x=(x)^+ +(x)^-.$ 

The stochastic process $X$ is called an SHP if the corresponding $\rho$ and $S$ satisfy the WHF. We propose the following definition of the WHF on graphs, which is a special case of the general formulation in~\cref{eq:WHs0}.

\begin{proposition}[Wasserstein Hamiltonian Flow on Graphs]\label{def:WHSoG}
A WHF defined on the cotangent bundle of $\mathcal{P}(\mathsf{G})$ is 
\begin{equation}\label{eq:WHs}
 \begin{cases}
 \frac{d\rho_i}{dt} = \sum_{j \in N(i)} [(S_i - S_j)^- \rho_i + (S_i - S_j)^+ \rho_j], \\
 \frac{d S_i}{dt} = -\frac{1}{2} \sum_{j \in N(i)}((S_j - S_i)^+)^2 - f_i(\rho, t),
 \end{cases}
 \end{equation}
 where the corresponding Hamiltonian is 
 \begin{equation}\label{eq:sHP}
 \mathcal{H}(\rho,S)
 =\frac{1}{4}\sum_{i\in V} \sum_{j\in N(i)}(S_i-S_j)^2 c_{ij}(\rho) + \mathcal{F}(\rho,t).
 \end{equation} 
Here, $\mathcal{F}(\rho,t)$ is a function depending on $\rho$ and $t$, and $f_i(\rho, t)=\frac{\partial}{\partial \rho_i}\mathcal{F}(\rho,t).$
\end{proposition}

\begin{proof}
 The Hamiltonian above~\cref{eq:sHP} can be expanded to $$\mathcal{H}(\rho,S)=\frac{1}{4}(\sum_{(i\rightarrow j) \in E}w_{ij}(S_i-S_j)^2 c_{ij}(\rho)+\sum_{(j\rightarrow i) \in E}w_{ji}(S_i-S_j)^2 c_{ji}(\rho)) + \mathcal{F}(\rho,t).$$ 
The discrete master equation is from $\frac{d \rho_i}{dt} = \frac{\partial \mathcal{H}}{\partial S_i}:$ 
$$
 \begin{aligned}
  \frac{\partial \mathcal{H}}{\partial S_i}
  &=\frac{1}{4}\left(\sum_{(i\rightarrow j) \in E} w_{ij}2(S_i - S_j)c_{ij}(\rho) + \sum_{(j\rightarrow i) \in E} w_{ji}2(S_i - S_j)c_{ji}(\rho)\right)\\
  &=\sum_{j \in N(i)}w_{ij}(S_i - S_j)c_{ij}(\rho)\\
  &=\sum_{j \in N(i)}w_{ij}(S_i - S_j)(\rho_i \mathbb{I}_{\{S_j > S_i\}} + \rho_j \mathbb{I}_{\{S_i > S_j\}}) \\
  &=\sum_{j \in N(i)} w_{ij}[(S_i - S_j)^-\rho_i + (S_i - S_j)^+\rho_j], 
 \end{aligned}
  $$ 
  and the Hamilton-Jacobi (HJ) equation is from the $\frac{d S_i}{dt} = -\frac{\partial \mathcal{H}}{\partial \rho_i}:$
  $$
  \frac{\partial \mathcal{H}}{\partial \rho_i} = \frac{1}{2} \sum_{j \in N(i)}w_{ij} ((S_j - S_i)^+)^2+f_i(\rho, t).
  $$
\end{proof}

We refer to the general SHP from the viewpoint of Lagrangian framework proposed in~\cite{cui2021stochastic}.

\section{Two-Point Boundary Stochastic Hamiltonian Process on Graphs}\label{sec:2ptBdySHP}

The two-point boundary problem for an SHP on a graph is generally difficult to solve directly if the graph is large. We therefore use an equivalent variational characterization known as the minimal action functional on the graph probability manifold $\mathcal P(G)$. 

\begin{definition}[Minimal Action Functional]\label{def:ActionFunctional}
Given an initial density $\rho(0)=\rho_0$, a kinetic energy $\mathcal{K}(\rho,S)=\frac{1}{2}\langle \nabla_G S,\nabla_G S \rangle_{c_{ij}(\rho)},$ a running potential $\mathcal{F}(\rho)$, and a terminal potential $\mathcal{G}(\rho(T))$, we define the constrained action functional as
\begin{equation}\label{eq:ActionFunctional}
 \inf_{\rho,S} \quad \mathcal{J}(\rho,S) = \int_0^T \left( \frac{1}{2}\langle \nabla_G S,\nabla_G S \rangle_{c_{ij}(\rho)} - \mathcal{F}(\rho) \right) dt + \mathcal{G}(\rho(T)),
\end{equation}
\begin{equation}\label{eq:continuity}
  \begin{aligned}
    \text{s.t.}\qquad &\frac{d\rho_i}{dt}=\sum_{j \in N(i)} \left( (S_i-S_j)^+ \rho_j + (S_i-S_j)^-\rho_i \right), \quad \text{for } i \in V, \\
    &\rho(0) = \rho_0.
 \end{aligned}
\end{equation}
\end{definition}
\Cref{eq:continuity} is the continuity equation, which is exactly the master equation $\frac{d \rho}{dt}=\rho Q$ utilized in our framework.

The first-order optimality conditions of the constrained action functional in~\cref{eq:ActionFunctional} is 
\begin{equation}\label{eq:ActionFunc1stOpt}
 \begin{cases}
 \frac{d \rho_i}{dt} +\sum_{j \in N(i)}(S_i - S_j) c_{ij}(\rho) = 0, \\
 \frac{d S_i}{dt} + \frac{1}{2} \sum_{j \in {N(i)}} (S_i-S_j)^2\frac{\partial{c_{ij}(\rho)}}{\partial \rho_i} + \frac{\partial \mathcal{F}}{\partial \rho_i} = 0, \\ 
 \rho_i(0) = (\rho_0)_i, \quad S_i(T) = -\frac{\partial \mathcal{G}}{\partial \rho_i}(T), \quad \text{for } i\in V. 
 \end{cases}
\end{equation}
While the core dynamics are known to arise directly from the partial derivatives of the Hamiltonian~\cref{eq:sHP}, introducing the unconstrained Lagrangian $\mathcal{L} = \mathcal{J} + \langle S, \frac{d \rho}{dt} - \rho Q \rangle$ allows us to capture the full boundary value problem. Finding the stationary point $\delta\mathcal{L} = 0$ recovers the symplectic Hamiltonian system~\cref{eq:WHs}, subject to the specific boundary conditions $\rho(0) = \rho_0$ and $S(T) = -g(\rho(T))$.

\subsection{Particle Level Formulations}
In this section, we present the particle-level characterization of the Hamiltonian system~\cref{eq:WHs}. Consider a single particle following the Markov process $\{X(t)\}_{t\geq 0}$, which is driven by the pair $(X, Q)$. At the microscopic level, the integrand of the minimal action functional corresponds to the particle-level Lagrangian $L(X, S)$. The Hamiltonian $H(X,S)$ is obtained via the Legendre transform of the Lagrangian $L(X, S)$ with respect to the potential $S$.
The macroscopic Hamiltonian $\mathcal{H}(\rho,S)$ is then exactly recovered by taking the expectation of $H(X,S)$ over the density $\rho$: 
\begin{equation}\label{eq:Hamiltonian-quadratic}
\begin{aligned}
  \mathcal{H}(\rho,S)
  &=\mathbb{E}_{X\sim \rho}\left[\frac{1}{4}\sum_{j\in N(X)}(S_X-S_j)^2 + F(X,\rho)\right]\\
  &= \mathbb{E}_{X\sim \rho}[H(X,S)],
\end{aligned}
\end{equation}
where $\mathcal{F}(\rho)=\mathbb{E}_{X\sim \rho}[F(X,\rho)]$ connects the microscopic and macroscopic costs.

\begin{definition}[Minimal Integral Functional at the Particle Level]\label{def:ActionFunctionalP}
Consider the action functional constrained by a time-inhomogeneous Markov process $X$:
\begin{equation}\label{eq:ActionFunctionalP}
 \begin{aligned}
 \inf_{S} \quad & \mathcal{J}(S) = \mathbb{E}_{X\sim \rho}\left[ \int_0^T \left( \frac{1}{4}\sum_{j\in N(X)} (S_X-S_j)^2 - F(X,\rho) \right) dt + G(X(T),\rho(T)) \right], \\
 \text{s.t.} \quad 
 & P(X(t + dt) = s_j \mid X(t) = s_i) = \delta_{ij} + Q_{ij}dt + o(dt), \quad \text{as } dt \to 0,
 \end{aligned}
\end{equation}
where $\rho = \operatorname{Law}(X)$ with initial density $\rho(0) = \rho_0$, and $X(0) \sim \rho_0$. 
Here, the transition rate matrix $Q$ is determined by $S$ according to~\cref{eq:parameterized-Q}. 
\end{definition}

Here, the forward particle trajectories of the Markov process $\{X(t)\}_{t\geq 0}$ evolve according to the transition rates $Q$ induced by $S$, allowing us to recover the probability distribution $\rho$ as the empirical distribution of these particles.

\subsection{Parameterization of the Stochastic Hamiltonian Process}

By parameterizing the potential function as $S^\theta$, we restrict the infinite-dimensional variational problem to a specific parametric function class, transforming it into a parameter optimization problem. 

The parameterized action functional, subject to a time-inhomogeneous Markov process constraint, is defined as:
\begin{equation}\label{eq:ParameterizedActionFunctionalP}
\begin{aligned}
 \inf_\theta \quad & \mathcal{J}(\theta) = \mathbb{E}_{X\sim \rho}\left[ \int_0^T \left( \frac{1}{4}\sum_{j\in N(X)}(S^\theta_X-S^\theta_j)^2 - F(X,\rho) \right) dt + G(X(T),\rho(T)) \right], \\
 \text{s.t.} \quad 
 & P^\theta(X(t + dt) = s_j \mid X(t) = s_i) = \delta_{ij} + Q^\theta_{ij}dt + o(dt), \quad \text{as } dt \to 0,
\end{aligned}
\end{equation}
where $\rho = \operatorname{Law}(X)$ with initial density $\rho(0) = \rho_0$, and $X(0) \sim \rho_0$. 
Here, the transition rate matrix $Q^\theta$ is determined by $S^\theta$ according to
\begin{equation}\label{eq:Q-rate-matrix}
Q^{\theta}_{ij} = 
\begin{cases}
 (S^\theta_j-S^\theta_i)^+, & j \in N(i), \\
 -\sum_{k\in N(i)} Q^\theta_{ik}, & i = j, \\ 
 0, & j \notin N(i) \cup \{i\}.
\end{cases}
\end{equation}

To minimize the parameterized action functional $\mathcal{J}(\theta)$, we update the parameters $\theta$ using gradient descent: 
\begin{equation}\label{eq:gradient-descent-update}
\theta^{[k+1]} := \theta^{[k]} - \eta \nabla_\theta \mathcal{J}(\theta^{[k]}),
\end{equation}
where $k$ is the iteration index and $\eta > 0$ is the learning rate. Upon convergence to the parameters $\theta^*$, we obtain the optimized potential $S^{\theta^*}$. By simulating the forward particle trajectories of the Markov process $\{X(t)\}_{t\geq 0}$ using the transition rate matrix $Q^{\theta^*}$ induced by $S^{\theta^*}$, we recover the time-dependent density $\rho(t)$ as the empirical distribution of these particles.

This particle-level forward simulation effectively solves the SHP
\begin{equation}\label{eq:pSHP}
  \begin{cases}
    P^{\theta^*}(X(t + dt) = j \mid X(t) = i) = \delta_{ij} + Q^{\theta^*}_{ij} dt + o(dt), \quad j \in N(i), \\[1mm]
    S^{\theta^*}(t+dt) = S^{\theta^*}(t) - \frac{\partial\mathcal{H}}{\partial \rho}dt + o(dt),
  \end{cases}
\end{equation}
while satisfying the boundary conditions $\rho_i(0) = (\rho_0)_i$ and $S_i(T) = -g_i(\rho(T))$ for $i \in V$.

\section{Numerical Method}\label{sec:methodology}
In this section, we present our implementation of the parameterized SHP at the particle level, where the action functional is evaluated using Monte Carlo methods.

\subsection{Discretization and Markov Decision Process}
The particle-level objective functional in~\cref{eq:ParameterizedActionFunctionalP} is computed over the entire trajectory on the graph and can be evaluated via time discretization.

The minimization of the action functional constrained by the particle dynamics can be regarded as a Markov decision process (MDP).
Consider the time discretization, denote $X^{l,(m)}:=X^l(t_m)$ as the state of particle $l$ at time step $m$ and $\rho^{(m)}:=\rho(t_m)$ the emperical distribution of the particles at time $m$. For a fixed particle $l,$ let the trajectory to be $$\gamma^l = \left(X^{l,(0)},X^{l,(1)},\cdots,X^{l,(M)}\right).$$ 
If $X^{l,(m)}=s_i,$ then its next state belongs to $\{s_i\} \cup \{s_j\}, {j\in N(i)}, i\in V.$ 

The particles jump discretely according to the transition probability 
\begin{equation}P^{\theta,h}(X^{l,(m+1)} = s_j \mid X^{l,(m)} = s_i) = \delta_{ij} + Q^\theta_{ij}h
\end{equation}
at each time step following~\cref{eq:ParameterizedActionFunctionalP}.

By the Markov property, the probability of a trajectory is \begin{equation}\label{eq:probability-P-gamma-theta}
 P^{\theta,h}(\gamma^l;X^{l,(0)})=\rho_0\prod_{m=0}^{M-1}P^{\theta,h}(X^{l,(m+1)}|X^{l,(m)}).
\end{equation}

\subsection{Estimation of the Gradients on Graphs and Update of Parameters}\label{sec:grad_estimation}
We need to update the parameter $\theta$ in the parameterized potential function $S^\theta$ according to the gradient descent process~\cref{eq:gradient-descent-update}. 
Evaluating the exact total gradient $\nabla_\theta \mathcal{J}(\theta)$ in~\cref{eq:gradient-descent-update} is computationally intractable because the selected node $X^{(m+1)}\in V$ depends discontinuously on $P^\theta$. Thus, the ordinary backpropagation cannot differentiate through the sampled trajectory. We design an estimation method rather than using automatic differentiation directly.

The gradient of the objective functional of one particle can be evaluated using the log-derivative identity \textit{i.e.} $\nabla_\theta \log P^{\theta,h}(\cdot) = \frac {\nabla P^{\theta,h}(\cdot)}{P^{\theta,h}(\cdot)}$. This method has been applied on score function estimator and reinforcement learning~\cite{barto2018reinforcement,williams1992simple}.

\begin{proposition}[Gradient Estimation on Graphs]\label{pro:grad_estimation}
The gradient of the objective functional of a single particle starting from $X(0)=s^{(0)}$ is
\begin{equation}\label{eq:gradJ}
\nabla_{\theta}J(\theta; s^{(0)})=\mathbb{E}_{\gamma\sim P^{\theta,h}(\cdot;s^{(0)})}\left[C^\theta(\gamma) \nabla_\theta \log P^{\theta,h}(\gamma ; s^{(0)}) + \nabla_\theta C^\theta(\gamma)\right],
\end{equation}
where the total accumulated cost of each sampled trajectory is: 
\begin{equation}\label{eq:C_gamma}
 C^\theta(\gamma)=G(X^{(M)},\rho^{(M)})+\sum_{m=0}^{M-1} \left[\sum_{j\in N(X^{(m)})}\frac{1}{4}(S^{\theta,(m)}_{X^{(m)}}-S^{\theta,(m)}_j)^2-F(X^{(m)},\rho^{(m)})\right] \Delta t.
\end{equation}

\end{proposition}
\begin{proof}

The objective functional of the single particle is
\begin{equation}\label{eq:J_X0}
 J(\theta; s^{(0)}) = \mathbb{E}_{\gamma \sim P^{\theta,h}(\cdot;s^{(0)})} [C^\theta(\gamma)|X(0)=s^{(0)}].
\end{equation}

As the gradient of $C^\theta(\gamma)$ cannot be directly calculated on discrete graphs,~\cref{eq:J_X0} can be evaluated by:
\begin{equation}\label{eq:grad_theta_J_derivation}
 \begin{aligned}
 \nabla_\theta J(\theta; s^{(0)}) &= \nabla_\theta \mathbb{E}_{\gamma \sim P^{\theta,h}(\cdot;s^{(0)})} [C^\theta(\gamma)]\\
 &= \nabla_\theta \left[ \sum_{\gamma} P^{\theta,h}(\gamma ; s^{(0)})\, C^\theta(\gamma) \right] \\
 &= \sum_{\gamma} C^\theta(\gamma)\nabla_\theta P^{\theta,h}(\gamma) + P^{\theta,h}(\gamma ; s^{(0)}) \, \nabla_\theta C^\theta(\gamma) \\
 &= \sum_{\gamma} P^{\theta,h}(\gamma ; s^{(0)}) \left[C^\theta(\gamma)\nabla_\theta \log P^{\theta,h}(\gamma ; s^{(0)}) \, +\nabla_\theta C^\theta(\gamma)\right]\\
 &= \mathbb{E}_{\gamma \sim P^{\theta,h}(\cdot;s^{(0)})} \left[ C^\theta(\gamma)\nabla_\theta \log P^{\theta,h}(\gamma ; s^{(0)}) + \nabla_\theta C^\theta(\gamma) \right].
 \end{aligned}
\end{equation}

\end{proof}

Known that the time-discretization of~\cref{eq:ParameterizedActionFunctionalP} is 
\begin{equation}\label{eq:mathcal-J}
 \mathcal{J}(\theta)=\mathbb{E}_{s^{(0)} \sim \rho_0}[J(\theta; s^{(0)})],
\end{equation}
and the gradient is 
\begin{equation}\label{eq:grad_mathcal-J}
 \nabla_\theta\mathcal{J}(\theta)=\mathbb{E}_{s^{(0)} \sim \rho_0}[\nabla_\theta J(\theta; s^{(0)})].
\end{equation}
We estimate the gradient by Monte-Carlo method.




\begin{corollary}
Letting $\widehat{\nabla_\theta \mathcal{J}}(\theta)$ be the Monte-Carlo estimator of the cost gradient $\nabla_\theta \mathcal{J}(\theta)$, we have
\begin{equation}\label{eq:grad-mathcalJ}
 \begin{aligned}
 \widehat{\nabla_\theta \mathcal{J}}(\theta)=\frac{1}{N}\sum_{l=1}^N \left(\sum_{m=0}^{M-1}C^\theta(\gamma^l)\nabla_\theta \log 
 P^{\theta,h}({X^{l,(m+1)}|X^{l,(m)}})\, 
 + \nabla_\theta C^\theta(\gamma^l)\right).
 \end{aligned}
\end{equation}
\end{corollary}
\begin{proof}

Taking gradient of both side of~\cref{eq:mathcal-J}, we have
\begin{equation}\label{eq:grad_mathcal_J}
 \nabla_{\theta}\mathcal{J}(\theta)
 =\mathbb{E}_{X(0) \sim \rho_0,\gamma \sim P^{\theta,h}} \left[C^\theta(\gamma) \nabla_\theta \log P^{\theta,h}(\gamma; X(0)) + \nabla_\theta C^\theta(\gamma) \right].
\end{equation}
Consider a trajectory starting from $ X^{l,(0)}=s^{l,(0)}$. Taking logarithm on both sides of~\cref{eq:probability-P-gamma-theta}, it becomes:
\begin{equation}\label{eq:probability-logP-gamma-theta}
\log P^{\theta,h}(\gamma^l;X^{l,(0)})=\log \rho_0 (X^{l,(0)}) + \sum_{m=0}^{M-1} \log P^{\theta,h}({X^{l,(m+1)}|X^{l,(m)}}).
\end{equation} The score function can be obtained by taking derivative of both side of~\cref{eq:probability-logP-gamma-theta}: 
\begin{equation}\label{eq:probability-gradlogP-gamma-theta}
 \nabla_\theta \log P^{\theta,h}(\gamma^l; X^{l,(0)})=\sum_{m=0}^{M-1} \nabla_\theta \log P^{\theta,h}({X^{l,(m+1)}|X^{l,(m)}}).
\end{equation}
Here, the first term of the right hand side $\nabla \log \rho_0 (X^{l,(0)})=0$ has been removed.

Combining~\cref{eq:grad_theta_J_derivation} and~\cref{eq:grad_mathcal_J}, the corresponding Monte-Carlo estimator of the gradient $\nabla_\theta \mathcal{J}(\theta^{[k]})$ is given by:
\begin{equation}
 \begin{aligned}
 \widehat{\nabla_\theta \mathcal{J}}(\theta^{[k]})= \frac{1}{N}\sum_{l=1}^N \left(\sum_{m=0}^{M-1}C^\theta(\gamma^l)\nabla_\theta \log P^{\theta,h}({X^{l,(m+1)}|X^{l,(m)}})\, 
 + \nabla_\theta C^\theta(\gamma^l)\right).
 \end{aligned}
\end{equation}

\end{proof}

Denote $\theta^{[k]}$ as the parameters set at the $k$th iteraten. The parameters set is updated as \begin{equation}\label{eq:parameter_update}\theta^{[k+1]} := \theta^{[k]} - \eta \widehat{\nabla_\theta \mathcal{J}}(\theta^{[k]}).\end{equation} Here, $\eta > 0$ is the learning rate.

\begin{algorithm}[htbp]
 \caption{Parameterized Stochastic Hamiltonian Process with Two Point Boundary conditions on Finite Graphs}\label{alg:neuralSHJP}
 \begin{algorithmic}
 \STATE{Given initial distribution ${\rho_0}$, sample $X(0)\sim {\rho_0}$, initialize the parameterized potential $S_i:=S(s_i,t;\theta)$ for $i \in V.$}
 \FOR{$k=0,\cdots,K$ epochs}
 \FOR{$m=0,\cdots,M-1$ steps$, \forall i \in V$}
 \FOR{each particle in the batch with $N$ particles}
 \STATE{Update the transition-rate matrix ${Q}$ at time and the transition probability ${P^{h}}$ of the particle $l$ at $X^{l,(m)}$, and sample the next state:}
 
 \STATE{\textbf{Denote the current state as $s^{l,(m)}=s_i,$} 
 $$
 \begin{cases}
 Q_{ij} \leftarrow (S(s_j,t_m;\theta)-S(s_i,t_m;\theta))^+, \quad j \in N(i)\\
 Q_{ii} \leftarrow (-\sum_{j\in N(i)}Q_{ij})^+, \quad j \in N(i)\\
 P_{ij} \leftarrow hQ_{ij}, \quad j \in N(i)\\
 P_{ii} \leftarrow 1+hQ_{ii},\\
 s^{l,(m+1)}\leftarrow \text{Sample} \left(X^{l,(m+1)} \sim P^{h}_{s^{l,(m)} j}\right).
 \end{cases}$$
 }
 \ENDFOR
 \STATE{Collect the trajectory of the state $s^{l,(\cdot)},$ parameterized potential funciton $S(s^{l,(\cdot)};\theta)$: $\gamma^l \leftarrow (s^{l,(\cdot)},{S(s^{l,(\cdot)};\theta)}).$}
 \ENDFOR

 Evaluate the gradient of the objective function estimated by~\cref{eq:grad-mathcalJ} corresponding to the trajectories $\gamma$ as

$$
  \widehat{\nabla_\theta \mathcal{J}}(\theta^{[k]})=\frac{1}{N}\sum_{l=1}^N \left(\sum_{m=0}^{M-1}{C^\theta(\gamma^l)} \nabla_\theta\log P^{\theta,h}({X^{l,(m+1)}|X^{l,(m)}}) + \nabla_\theta C^\theta(\gamma^l)\right).
$$

Update the potential function ${S}(V;\theta^{[k+1]})$ according to~\cref{eq:parameter_update} by: $$\theta^{[k+1]} := \theta^{[k]} - \eta \nabla_{\theta}\widehat{\nabla_\theta \mathcal{J}}(\theta^{[k]}).$$
 \ENDFOR
 
 \RETURN $S^* = {S}(\theta^*)$
 \end{algorithmic}
 \end{algorithm}

 \subsection{Algorithm}

The parameterized SHP algorithm is summarized in \cref{alg:neuralSHJP}. At time step $m$, the graph-valued random variable $X^{l,(m)}$ denotes the state of particle $l$. Conditional on
$
X^{l,(m)}=s_i,
$
the current-node index is $i$, and the candidate next-node index satisfies $j\in N(i)\cup\{i\}$. The potential network is evaluated using the node representation $s_i$ and time $t_m$.

\subsection{Scalability Analysis}
In this section, we analyze the computational complexity and memory storage. 

The degree of the graph on each node is $d_i, i=1,\cdots,n,$ and the number of edges is $|E|.$ As the total degree of all the nodes on the graph is $\sum_{i \in V}d_i=2|E|\leq|V|d_\text{max},$ the number of edges are bounded as $|E| \leq \frac{|V|d_\text{max}}{2}.$

\textbf{Computational complexity.}
Let $N$, $M$ and $K$ denote the numbers of particles, time steps and optimization iteratens, respectively. We assume fixed feature dimension and network architecture.

In each particle transition, at most $d_{\max}+1$ candidate states are considered. Let $C_{\mathrm{pint}}$ and $C_{\mathrm{pter}}$ denote the computational costs of the interaction and terminal terms, including differentiation. The particle formulation therefore has computational complexity
\begin{equation}\label{eq:complexity-particle}
\mathcal{C}_{\mathrm{SHP}}
=
\mathcal{O}\!\left(
K\left[
M\left(N(d_{\max}+1)+C_{\mathrm{pint}}\right)
+C_{\mathrm{pter}}
\right]
\right).
\end{equation}
Graph-neighbor interactions can be evaluated in $\mathcal{O}(N(d_{\max}+1))$ operations when node masses are efficiently aggregated and accessed, whereas direct all-pairs interactions require $\mathcal{O}(N^2)$ operations. For the latter, if the terminal cost does not dominate,
$$
\mathcal{C}_{\mathrm{SHP}}
=
\mathcal{O}\!\left(KMN(N+d_{\max}+1)\right).
$$

For an edge-based sparse implementation, the density formulation has computational complexity
\begin{equation}\label{eq:complexity-fluid}
\mathcal{C}_{\mathrm{WHF}}
=
\mathcal{O}\!\left(
K\left[
M\left(|V|+|E|+C_{\mathrm{dint}}\right)
+C_{\mathrm{dter}}
\right]
\right),
\end{equation}
where $C_{\mathrm{dint}}$ and $C_{\mathrm{dter}}$ are the corresponding interaction and terminal costs. Edge-local and direct all-pairs interactions require $\mathcal{O}(|E|)$ and $\mathcal{O}(|V|^2)$ operations,
respectively.

Thus, our particle formulation SHP replaces full-graph evolution with computations on sampled particles and their local neighborhoods. The computational cost in SHP depends on $N$ and $d_{\max}$ rather than directly on all nodes $|V|$ and edges $|E|$ in density-based WHF, provided that the interaction and terminal evaluations do not require full-density computation. Thus, our SHP method has smaller computational complexity than WHF on large-scale graphs. 

\textbf{Memory complexity.}
Storing $D$ features per node, edge indices, and padded neighbor indices requires
$$
\mathcal{M}_G
=
\mathcal{O}\!\left(
D|V|+|E|+|V|d_{\max}
\right).
$$
For fixed $D$ and bounded $d_{\max}$, this reduces to $\mathcal{O}(|V|)$.

Let $R_{\mathrm{pint}}$, $R_{\mathrm{pter}}$,
$R_{\mathrm{dint}}$, and $R_{\mathrm{dter}}$ denote the additional memory costs of the corresponding evaluations, including tensors saved for backpropagation. Retaining the complete differentiated trajectories gives the memory bounds in the particle-based SHP method $$
\mathcal{M}_{\mathrm{SHP}}
=
\mathcal{M}_G
+
\mathcal{O}\!\left(
M\left[N(d_{\max}+1)+R_{\mathrm{pint}}\right]
+R_{\mathrm{pter}}
\right)
$$
and density-based WHF method
$$
\mathcal{M}_{\mathrm{WHF}}
=
\mathcal{M}_G
+
\mathcal{O}\!\left(
M\left[|V|+|E|+R_{\mathrm{dint}}\right]
+R_{\mathrm{dter}}
\right).
$$

In addition to the common one-time graph memory $\mathcal{M}_G$, the memory storage in SHP depends on $MN$ and $Md_{\max}$ rather than directly on all nodes $M|V|$ and edges $M|E|$ in density-based WHF along the trajectory, thus our SHP method require much smaller memory storage than WHF on large-scale graphs. 

\begin{remark}[Mini-batch interaction]
In the implementation of our SHP method, evaluating the exact interaction at a single time step with batch size $B$ requires a pairwise memory of $R_{\mathrm{pint}}=\mathcal{O}(NB)$ per time step, down from $\mathcal{O}(N^2)$, while still preserving the global interaction. In contrast, WHF cannot achieve a similar reduction in memory cost, as it inherently requires the full density profile at each time step.
\end{remark}

\section{Experiments}\label{sec:experiments}

This section evaluates the empirical performance of our proposed method, Neural SHP. We first outline the implementation details, where a neural network parameterizes the potential function $S(i,t)$, and the parameters are optimized via our derived gradient estimation method. Subsequently, we demonstrate the effectiveness of Neural SHP across synthetic experiments and real-world graph datasets. Where applicable, we implement the WHF on graphs~\cref{eq:ActionFunctional}, termed Neural WHF, as a baseline for comparison. Neural WHF is a density-based method that directly computes node-wise densities.

\textbf{Neural parameterized potential function.}
We parameterize the potential function $S(i, t) = S_{NN}^\theta(i, t)$ using multi-layer perceptron (MLP).
For the particles on the graph, let $X^l_i$ denote the $l$th particle on node $s_i$. The $(D+1)$-dimensional input to the neural network is given by the concatenated vector $\mathbf{X}_i^{t} = (X_i,t)^\top \in \mathbb{R}^{D+1}$, which yields the output $S^\theta_{NN} (\mathbf{X}_i^{t}) = S(i,t)$ for all $i \in V$. Specifically, we employ a two-hidden-layer MLP formulated as: $${S}^\theta_{NN}(\mathbf{X}_i^{t}) = \mathbf{W}_{out}\cdot \sigma \left( \mathbf{W}_2 \cdot \sigma \left( \mathbf{W}_1 \cdot \mathbf{X}_i^{t} + \mathbf{b}_1\right) + \mathbf{b}_2 \right) + \mathbf{b}_{out},$$ where the activation function is $\sigma(\cdot) = \text{Tanh}(\cdot)$, and $\theta$ denotes the set of trainable weights ($\mathbf{W}$) and biases ($\mathbf{b}$). The hidden layers are configured with $H=64$ neurons.


\textbf{Transition probability modification.}
Across all our experiments, we must ensure valid transition probabilities ($P^\theta \geq 0$). To achieve this and prevent numerical underflow, we modify the transition probability $P_{ij}^{(m)}$ by applying a lower-bound threshold $\varepsilon$ to the unnormalized values $\widetilde{P}_{ij}^{(m)}$:$$P_{ij}^{(m)} = \frac{\max\{\widetilde{P}_{ij}^{(m)}, \varepsilon\}}{\sum_{k \in \{i\} \cup N(i)} \max\{\widetilde{P}_{ik}^{(m)}, \varepsilon\}}.$$
For numerical evaluation, we set this threshold to $\varepsilon = 10^{-14}$.
 
\textbf{Evaluation metrics.}
Define the time-discretized kinetic energy and potentials as follows:
\begin{equation}
\mathcal K_{\Delta t}
:=
\Delta t
\sum_{m=0}^{M-1}\mathcal K(t_m),
\qquad
\mathcal F_{\Delta t}
:=
\Delta t
\sum_{m=0}^{M-1}\mathcal F(t_m).
\end{equation}
$\mathcal F=\lambda_V \mathcal V + \lambda_W \mathcal W,$ $\lambda_V$ and $ \lambda_W$ are the coefficients. The discrete kinetic energy $\mathcal K_{\Delta t}$ and potential energy $\mathcal F_{\Delta t}$ are used to evaluate the results of the experiments. Across all experiments, models are evaluated at the training epoch that achieves the lowest estimated objective $\widehat{\mathcal{J}}^* = \frac{1}{2}\widehat{\mathcal{K}}_{\Delta t} -\lambda_F\mathcal F_{\Delta t} + \lambda_G\widehat{\mathcal{G}}_B$.
 


\subsection{Experiments on Synthetic Graphs}\label{subsec:testexample}


In this section, we conduct a series of experiments to evaluate the performance of our proposed Neural SHP method on synthetic graphs. By default, we use $N=500$ particles for the synthetic numerical experiments. 



\subsubsection{Computational complexity with respect to particle number on lattice graphs}\label{subsubsec:particle-number}

Let $\rho_\mathsf{G}(\boldsymbol{z}; \mathbf{m}, \mathbf{\Sigma})$ denote a $D$-variate Gaussian distribution defined over spatial coordinates $\boldsymbol{z}= (z_1, z_2, \ldots, z_D) \in \mathbb{R}^D$, with mean $\mathbf{m} \in \mathbb{R}^D$ and covariance matrix $\mathbf{\Sigma} \in \mathbb{R}^{D \times D}$. 
We compare the computational times of potentials $\mathcal V$ and $\mathcal W$ using $N$ particle sampled from $\rho(0)=\rho_\mathsf{G}\left(s;0.2\mathbf 1_2,0.05\mathbf{I}_2\right)$ on a $31\times31$-grid lattice graph. The linear potential is defined as $\mathcal{V} := \sum_{i\in V} \mathbb{V}_i\rho_i$ and the pairwise interaction potential as $\mathcal{W} := \frac{1}{2} \sum_{i\in V}\sum_{j\in V} \mathbb{W}_{ij}\rho_i\rho_j$. The symmetric positive-semidefinite interaction matrix $\mathbb{W}$ uses a Gaussian kernel to model the congestion penalty between particles:
\begin{equation}\label{eq:gaussian-kernel}
\mathbb{W}^{g}_{ij}
:=
\exp\left(
-\frac{\|X_i-X_j\|^2}{2\sigma_W^2}
\right);
\end{equation} 




To empirically verify the theoretical complexities ($\mathcal{O}(N)$ for linear terms and $\mathcal{O}(N^2)$ for interaction terms), Figure~\ref{fig:v-w-complexity} compares the execution time of evaluating the linear potential $\mathcal{V}$ against the global interaction potential $\mathcal{W}$. The measured runtimes exhibit an empirical scaling of $\text{Time}_\mathcal{V} \propto N^{0.973}$ and $\text{Time}_\mathcal{W} \propto N^{1.990}$. This confirms our analytical expectations: evaluating $\mathcal{V}$ scales linearly with $N$, whereas the global interaction $\mathcal{W}$ scales quadratically. Correspondingly, the analytical storage requirements are $\mathcal{O}(N)$ and $\mathcal{O}(N^2)$, respectively.

\begin{figure}[htbp]
 \centering
 \includegraphics[width=0.75\linewidth]{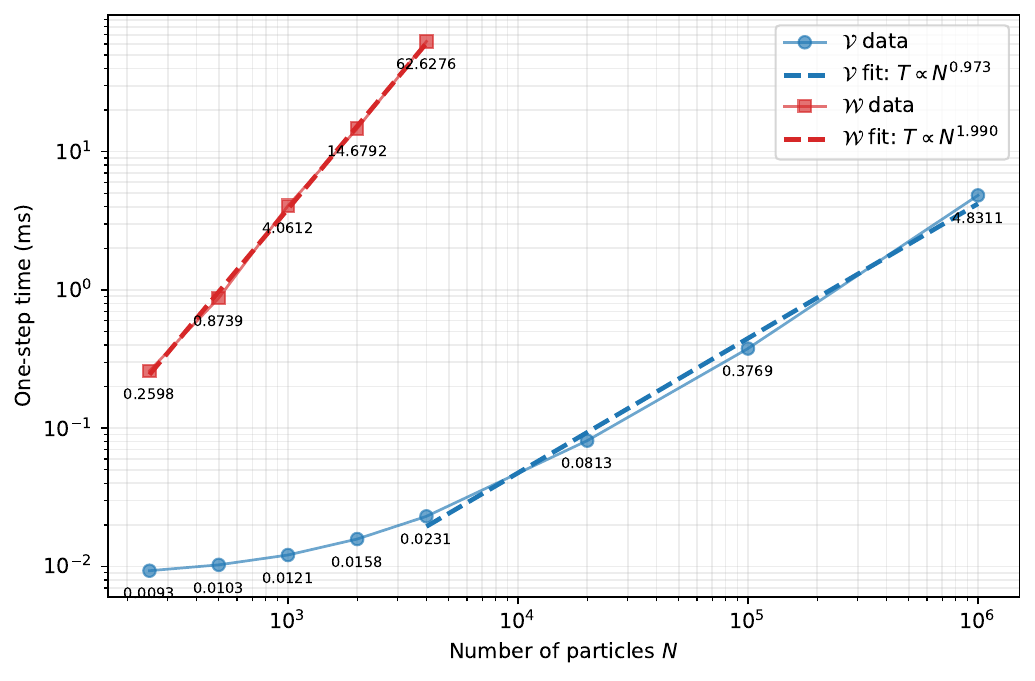}

\caption{Log-log scaling of computational time with respect to the number of particles $N$. The plot compares the linear potential $\mathcal{V}$ and the global interaction potential $\mathcal{W}$, with dashed lines indicating their respective least-squares fits.}
\label{fig:v-w-complexity}
\end{figure}

The dashed segments represent log-log least-squares fits applied to the five rightmost points for $\mathcal{V}$ and the four rightmost points for $\mathcal{W}$. Data points corresponding to smaller $N$ are excluded from the regression, as fixed computational overhead obscures the true asymptotic scaling.

\subsubsection{Scalability across graph dimensions}\label{subsubsec:graph-dim}
To efficiently scale the graph size, we construct regular lattices with eight nodes per coordinate axis and vary the coordinate dimension from $D=1$ to $7$ in the $[0,1]^D$ unit square. As $D$ increases, the total number of nodes and undirected edges scale exponentially:
$$|V|=8^D,
 \qquad
 |E|=7D 8^{D-1}.$$
The initial and target terminal distributions used in these experiments are $$
\rho_0(\boldsymbol{z})=\rho_\mathsf{G}\left(\boldsymbol z;\,0.2\mathbf 1_D,\,0.05\mathbf{I}_D\right),
\qquad
\rho_T(\boldsymbol{z})=\rho_\mathsf{G}\left(\boldsymbol z;\,0.8\mathbf 1_D,\,0.05\mathbf{I}_D\right).
$$

For these experiments, the temporal domain is discretized into $M=25$ steps with a step size of $\Delta t=0.5$, yielding a terminal time of $T=12.5$.

\begin{remark}[Implicit implementation]\label{remk:implicit-implementation}
On regular lattices, neighboring nodes can be computed dynamically from their indices, eliminating the need to store edges in memory.
\end{remark}

\begin{table}[htbp]
 \centering
 \small
 \begin{tabular}{@{}rrrrrrr@{}}
 \toprule
 $D$ & $|V|$ & $|E|$ & $\text{Time}_{\text{SHP}} (s)$ & $\text{Time}_{\text{WHF}} (s)$ & $R^\text{Peak}_{\text{SHP}} (\text{MB})$ & $R^\text{Peak}_{\text{WHF}} (\text{MB})$ \\
 \midrule
 1 & 8 & 7 & 0.0320 & 0.0111 & 356.20 & 350.22 \\
 2 & 64 & 112 & 0.0345 & 0.0137 & 346.66 & 347.71 \\
 3 & 512 & 1,344 & 0.0433 & 0.0285 & 357.36 & 354.10 \\
 4 & 4,096 & 14,336 & 0.0601 & 0.1187 & 378.43 & 432.12 \\
 5 & 32,768 & 143,360 & 0.0783 & 0.6358 & 423.10 & 1252.10 \\
 6 & 262,144 & 1,376,256 & 0.1121 & 6.2331 & 1313.13 & 5428.16 \\
 \bottomrule
 \end{tabular}\caption{Explicit implementation scaling results on lattices with eight nodes per axis, where $t$ is the epoch time (s) and $R^\text{Peak}$ is the peak memory (MB).}
 \label{tab:dimension-scalability}
\end{table}

\begin{figure}[htbp]
 \centering
 \includegraphics[width=0.7\linewidth]
 {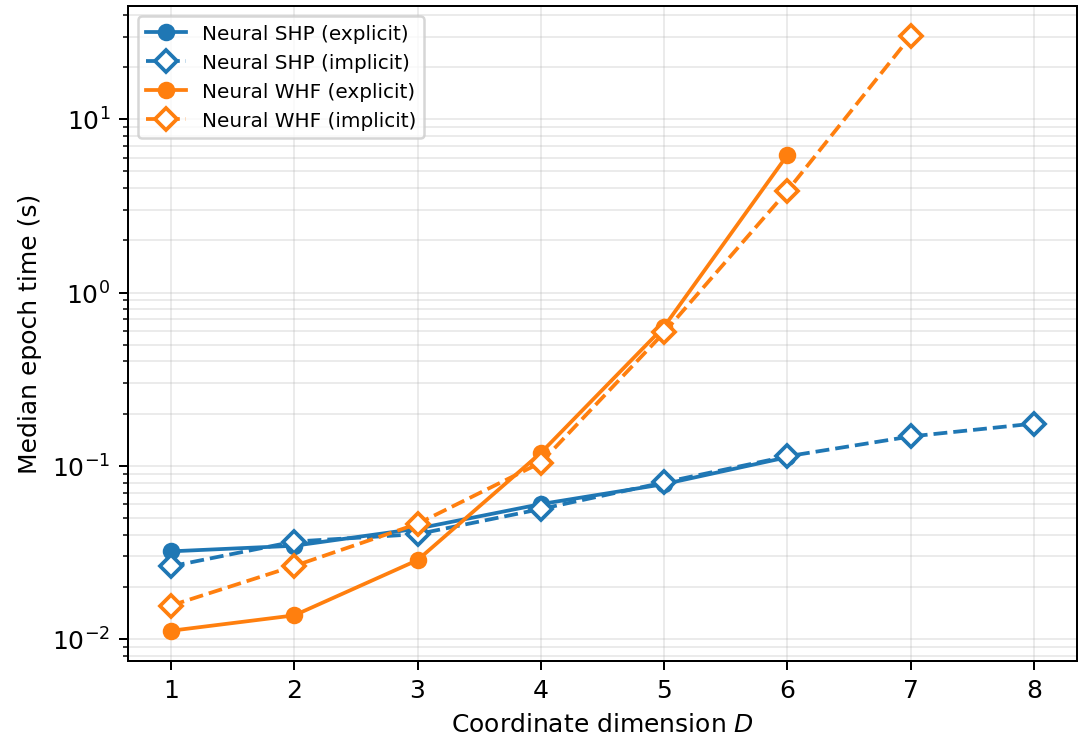}
 \caption{Log-scaled median computational time per epoch as a function of the coordinate dimension $D$.}
 \label{fig:computational_time_dim}
\end{figure}

Table~\ref{tab:dimension-scalability} and Figure~\ref{fig:computational_time_dim} demonstrate that Neural SHP scales significantly better than Neural WHF. Using explicit implementations, Neural WHF is faster for smaller graphs ($D \leq 3$), but Neural SHP overtakes it starting at $D=4$. This performance gap widens exponentially, rendering Neural SHP $55.6$ times faster at a coordinate dimension of $D=6$.

By leveraging the implicit implementation (Remark~\ref{remk:implicit-implementation}), denoted by dashed lines in Figure~\ref{fig:computational_time_dim}, we scale Neural SHP up to $D=8$. In contrast, Neural WHF was not executed at $D=8$. Differentiating an $M=25$-step density trajectory over all $8^8=16{,}777{,}216$ nodes requires an estimated peak memory of $214.75$ GB, vastly exceeding our $16$ GB hardware limit.

\subsubsection{Comparison between the particle-based and density-based methods on the lattice graphs across different grid sizes}\label{subsubsec:grid-size}
We compare the performance of Neural SHP and Neural WHF on two-dimensional lattice graphs of varying grid sizes. For this experiment, the spatial domain is defined as the square $\Omega=[0,1]^2$.

In these experiments, the initial distribution is a Laplacian distribution centered at $(0.3,0.3)^\top$:$$\rho_0(\boldsymbol{z}) = \frac{1}{Z_0} \exp\left( -\frac{\vert{}z_1-0.3\vert{} + \vert{}z_2-0.3\vert{}}{0.05} \right),$$where $Z_0$ is the normalization constant ensuring $\sum_{i\in V}\rho_i(0)=1$. The target distribution is $\rho_T(\boldsymbol{z}) = \rho_\mathsf{G}(\boldsymbol{z}; \,0.7\mathbf{1}_2, \,0.0025\mathbf{I}_2)$.

In this experiment, we use the squared-Maximum Mean Discrepancy ($MMD^2$) as the terminal condition of Neural SHP, which is calculated by the expectation of the single particle terminal costs. For a specific particle $l$, this terminal cost is given by:
\begin{equation}\label{eq:mmd}
 G_{\text{MMD}^2}(X^l(T),Y) = \frac{1}{N} \sum_{k=1}^N \mathbb{K}(X^l(T), X^k(T)) - \frac{2}{\tilde N} \sum_{k=1}^{\tilde N} \mathbb{K}(X^l(T), Y^k) + C,
\end{equation} 
where the kernel $\mathbb{K}$ here is set to be Gaussian kernel ($\mathbb{K}=\mathbb{W}^g$) with bandwidth $\sigma_{\text{MMD}}=\sigma_W$. In~\cref{eq:mmd}, $Y = \{Y^k\}_{k=1}^{\tilde N}$ represents $\tilde N$ samples from the target distribution $\rho_T$, and $C$ captures the self-interaction of these target samples. The terminal cost can also be regarded as a linear potential plus an interaction potential. We set a bandwidth parameter to be $\sigma_{\text{MMD}} = 0.2$ in this experiment. 

Taking the expectation over the terminal distribution yields the macroscopic potential, $\mathbb{E}_{X(T)\sim \rho(T)}[G_{\text{MMD}^2}] \propto \mathcal{G}_{\text{MMD}^2}(X,Y)$
\begin{equation}\label{eq:Emmd}
\mathbb{E}_{X(T)\sim \rho(T)}[G_{\text{MMD}^2}] = \frac{1}{N^2} \sum_{l=1}^N \sum_{k=1}^N \mathbb{K}(X^l(T), X^k(T)) - \frac{2}{N {\tilde N}} \sum_{l=1}^N \sum_{k=1}^{\tilde N} \mathbb{K}(X^l(T), Y^k).
\end{equation} 

For Neural WHF, we formulate the terminal loss using the $\ell_1$-distance between the computed terminal distribution ${\rho}(T)$ and the prescribed target distribution $\rho_T$:
\begin{equation}\label{eq:l1-distance}
{\mathcal{G}}_{\ell_1}=\left\|{\rho}(T)-\rho_T\right\|_1:=\sum_{i\in V}\left|{\rho}_{i}(T)-(\rho_T)_{i}\right|.
\end{equation}

The temporal domain is discretized into $M=50$ steps with a step size of $\Delta t=0.5$, yielding a terminal time of $T=25$. Both models are trained for $3000$ epochs. The loss weighting coefficients are setted to $\lambda_K : \lambda_G = 0.5 : 5000$
 Table~\ref{tab:grid-size-comparison-m50mmd02w005} summarizes the results. Notably, despite employing fundamentally different terminal potentials, both methods yield similar kinetic energy measurements. Furthermore, Neural SHP demonstrates superior computational scalability. Its epoch time grows significantly slower than that of Neural WHF as the grid expands.

\begin{table}[htbp]
 \centering
 \begin{minipage}[t]{0.5\linewidth}
 \centering
 \resizebox{\linewidth}{!}{%
 \begin{tabular}{cccc}
 \toprule
 Grid & ${\mathcal K}_{\Delta t}$ & ${\mathcal{G}}_{\text{MMD}^2}$ & Time/epoch (s) \\
 \midrule
 $21\times 21$ & 3.4684 & $3.0839\times10^{-4}$ & 0.0613 \\
 $31\times 31$ & 6.5186 & $3.7730\times10^{-4}$ & 0.0639 \\
 $41\times 41$ & 10.7460 & $7.2038\times10^{-4}$ & 0.0760 \\
 $51\times 51$ & 16.6010 & $1.4771\times10^{-3}$ & 0.0712 \\
 \bottomrule
 \end{tabular}
 }\subcaption{Neural SHP, terminal at ${\mathcal{G}}_{\text{MMD}^2}$}\label{subtab:neural-shp-mmd2}
 \end{minipage}\hfill
 \begin{minipage}[t]{0.45\linewidth}
 \centering
 \resizebox{\linewidth}{!}{%
 \begin{tabular}{cccc}
 \toprule
 Grid & $ {\mathcal K}_{\Delta t}$ & ${\mathcal{G}}_{\ell_1}$ & Time/epoch (s) \\
 \midrule
 $21\times 21$ & 3.8932 & 0.0143 & 0.0403 \\
 $31\times 31$ & 6.5828 & 0.0089 & 0.0587 \\
 $41\times 41$ & 10.8170 & 0.0103 & 0.0939 \\
 $51\times 51$ & 16.5670 & 0.0119 & 0.1297 \\
 \bottomrule
 \end{tabular}
 }\subcaption{Neural WHF, terminal at ${\widehat{\mathcal{G}}}_{\mathrm{\ell_1}}$}\label{subtab:neural-whf-l1}
 \end{minipage}
 \caption{Comparison of Neural SHP and Neural WHF across varying grid sizes.}
 \label{tab:grid-size-comparison-m50mmd02w005}
\end{table}

\subsection{Experiments on obstacles avoidance on graphs}\label{subsec:obstacle-avoidance}

We compare hard and soft obstacle representations on a common triangular base grid. We define the full domain $\Omega$, a circular barrier $\mathcal{B}_{\mathrm{c}}$, and a rectangular barrier $\mathcal{B}_{\mathrm{r}}$ as:$$\begin{aligned} \Omega &=\left\{\boldsymbol z\in\mathbb R^2: \left\lVert\boldsymbol z-\tfrac12\mathbf 1_2\right\rVert_2<\tfrac12\right\},\\ \mathcal B_{\mathrm c} &=\left\{\boldsymbol z\in\Omega: \left\lVert\boldsymbol z-\tfrac12\mathbf 1_2\right\rVert_2\leq R\right\},\\ \mathcal B_{\mathrm r} &=\left\{\boldsymbol z\in\Omega: z_1<0.2,\ \vert{}z_2-0.5\vert{}<0.1\right\}. \end{aligned}$$The hard obstacle model restricts the domain to $\Omega_{\mathrm{hard}} = \Omega\setminus(\mathcal B_{\mathrm c}\cup\mathcal B_{\mathrm r})$, explicitly removing nodes and incident edges within the barriers. Conversely, the soft obstacle model retains all nodes ($\Omega_{\mathrm{soft}}=\Omega$) but imposes a repulsive potential landscape $\mathbb{V}(\boldsymbol{z}) = \mathbb{V}_{\mathrm{c}}(\boldsymbol{z}) + \mathbb{V}_{\mathrm{r}}(\boldsymbol{z})$ to represent the barriers:$$\begin{aligned} \mathbb V_{\mathrm c}(\boldsymbol z) &=\frac{A}{2}\left[ 1+\tanh\!\left(\eta\left( R-\left\lVert\boldsymbol z-\tfrac12\mathbf 1_2\right\rVert_2 \right)\right)\right],\\ \mathbb V_{\mathrm r}(\boldsymbol z) &=\frac{A}{4} \left[1+\tanh\!\left(\eta(0.2-z_1)\right)\right] \left[1+\tanh\!\left(\eta(0.1-\vert{}z_2-0.5\vert{})\right)\right]. \end{aligned}$$


For the soft obstacles, we set $A=250$, $\eta=20$, $R=0.1$, and the potential weight $\lambda_V=0.5$. In both scenarios, the terminal penalty is governed by $\mathcal{G}_{\text{MMD}^2}$ with weight $\lambda_G=5000$ and bandwidth $\sigma_{\text{MMD}}=0.0333$. Both experiments simulate $M=40$ steps with $\Delta t=0.5$ (terminal time $T=20$). $N=500$ particles are used in Neural SHP. The hard- and soft-obstacle models are trained with learning rates of $10^{-3}$ and $2\times 10^{-4}$, respectively.

Figure~\ref{fig:re-linear} illustrates the learned particle densities and potential functions. In both cases, the population successfully avoids the obstacles by splitting around the center and re-aggregating near the target. Table~\ref{tab:obstacle-hard-soft} reports the optimal results. While the hard obstacle configuration achieves a lower terminal loss, the soft obstacle setup yields a lower discrete $\ell_1$ terminal discrepancy~\eqref{eq:l1-distance}.

\begin{table}[htbp]
 \centering
 \begin{tabular}{cccccccc}
 \toprule
 Constraint& $|V|$ & $|E|$ &${\mathcal K}_{\Delta t}$ &${\mathcal V}_{\Delta t}$ &${\mathcal{G}}_{\text{MMD}^2}$ &${\mathcal J}^*$ & $\|\rho(T)-\rho_T\|_1$ \\
 \midrule
 Hard & 752 & 2123 & 8.6758 & 0 & 0.0009 &50.7292 & 0.5965 \\
 \midrule
 Soft & 824 & 2369 & 9.8684 & 57.4285 & 0.0113 &119.0040 & 0.4866 \\
 \bottomrule
 \end{tabular}%
 \caption{Evaluation results of Neural SHP for both obstacle types.}
 \label{tab:obstacle-hard-soft}
\end{table}
 
\begin{figure*}[htbp] 
 \centering
 \begin{minipage}[b]{0.99\textwidth} 
 {
 \begin{minipage}[b]{0.49\linewidth} 
 \centering
 \includegraphics[width=\linewidth]{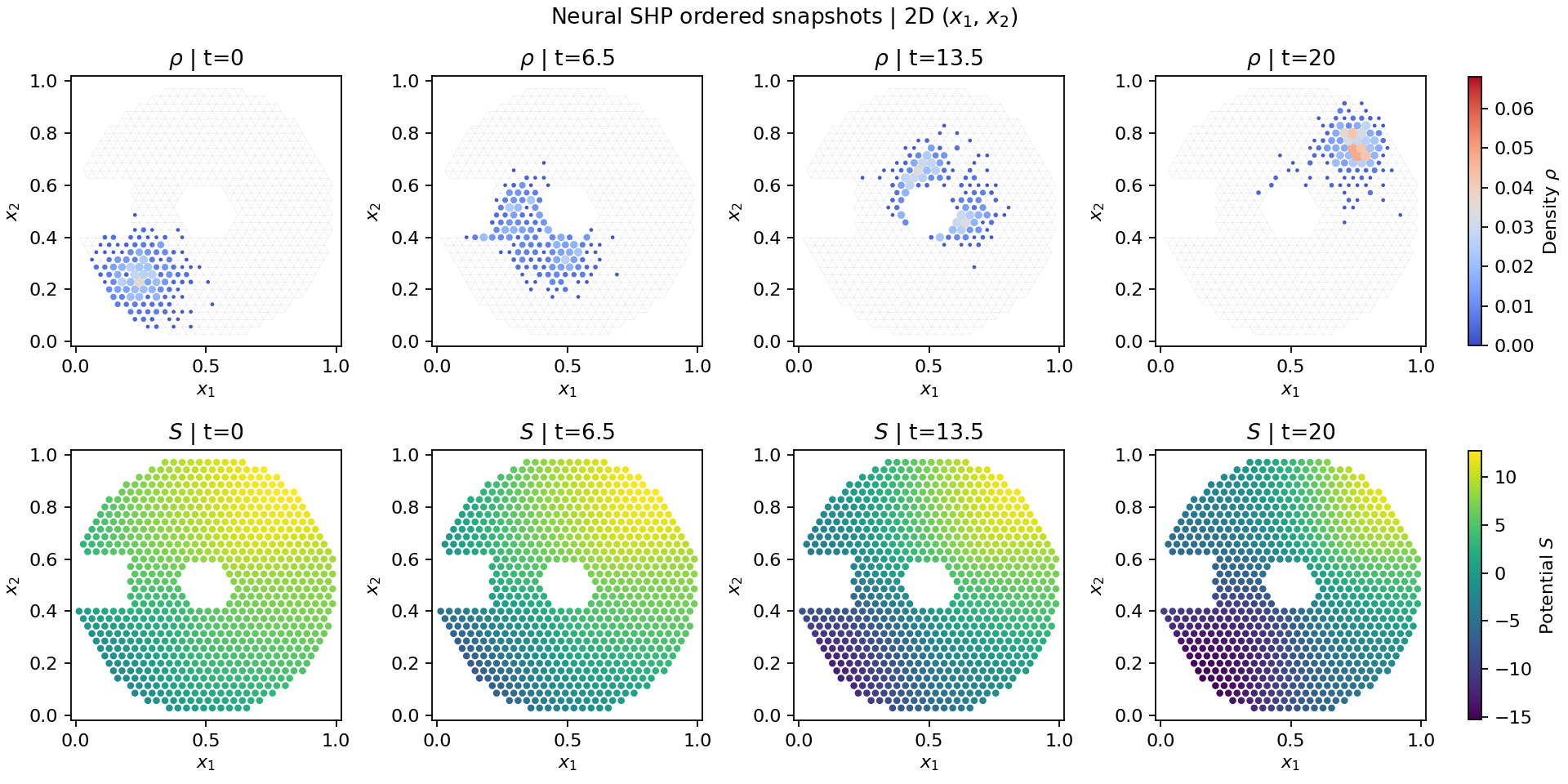}
 \subcaption{Hard obstacle snapshots}\label{fig:re-linear-hard}
 \end{minipage}
 }
 {
 \begin{minipage}[b]{0.49\linewidth}
 \centering
 \includegraphics[width=\linewidth]{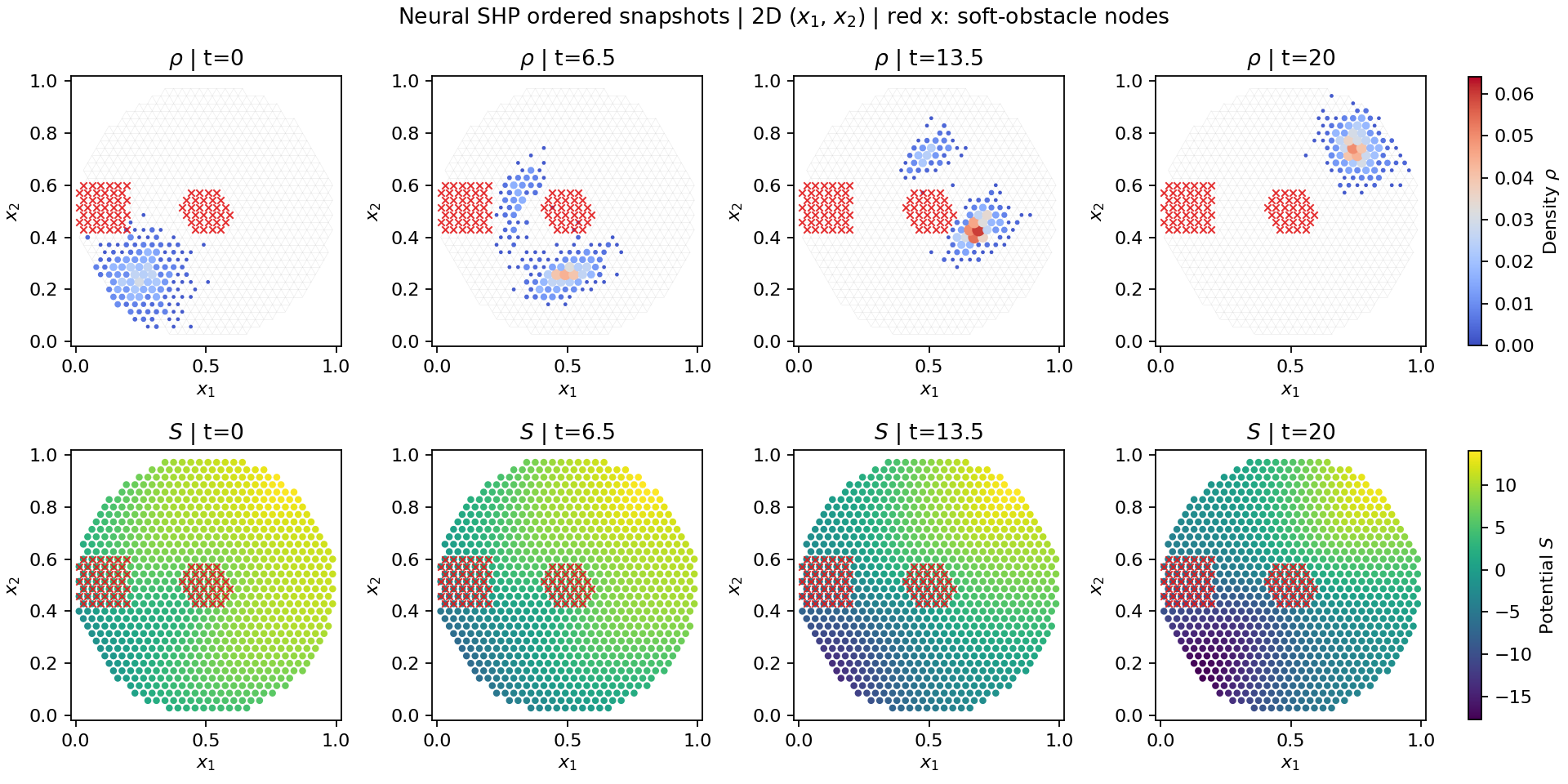}
 \subcaption{Soft obstacle snapshots}\label{fig:re-linear-soft}
 \end{minipage}
 }
 \end{minipage}
 \caption{
  Snapshots of the particle density dynamics $\rho$ (upper panel) and the learned potential function $S$ (lower panel) evaluated at time steps $t = 0, 6.5, 13.5,$ and $20$. Marker areas are scaled proportionally to nodal density. Panel (a) illustrates the hard obstacle scenario, while panel (b) depicts the soft obstacle scenario.
  }\label{fig:re-linear}
\end{figure*}

\subsection{Graph-distance transport on real networks}
\label{sec:real-graph-distance-transport}

The following real-world experiments are formulated as target-set transports from a source set $A \subseteq V$ to a target set $B \subseteq V$. For undirected edges $\{u,v\} \in E$ with positive route lengths $r_{uv} = r_{vu}$, the distance from node $i$ to $B$ is defined as $\operatorname{dist}(i,B) := \min_{b\in B} \operatorname{dist}(i,b)$, where $\operatorname{dist}(i,b) := \min_{\gamma:i\leadsto b} \sum_{\{u,v\}\in\gamma} r_{uv}$ is the shortest-path distance to a specific node $b$. We set $r_{uv}=1$ for unweighted graph-hop distances and use physical lengths for road networks.

Particles are initialized uniformly over the source set: $\rho_i(0) = \frac{\mathbb{I}_{\{i\in A\}}}{\vert{}A\vert{}}$. We define a normalized distance $G_B(i) = \frac{\operatorname{dist}(i,B)}{D_B}$ using a case-specific scale $D_B > 0$. 

In these real-world graph experiments, we use $N=500$ particles for training and $N_\text{eval}=10{,}000$ particles for evaluations. The terminal cost functional $\mathcal{G}_B(\rho(T)) = \sum_{i\in V} \rho_i(T) G_B(i)$ is empirically estimated with $N_{\mathrm{eval}}$ particles as:$$\widehat{\mathcal{G}}_B = \frac{1}{N_{\mathrm{eval}}} \sum_{l=1}^{N_{\mathrm{eval}}} G_B(X^l(T)).$$Transport success is quantified at $T$ by the target mass $$m_B(T) = \frac{1}{N_{\mathrm{eval}}} \sum_{l=1}^{N_{\mathrm{eval}}} \mathbb{I}_{\{X^l(T)\in B\}}$$ and the mean remaining distance to $B$ is ${\mathbb{E}}[\operatorname{dist}(X(T),B)] = D_B \widehat{\mathcal{G}}_B$. 

Noted that the metrics $\widehat {\mathcal K}_{\Delta t}, \widehat {\mathcal G}_B, \widehat {\mathcal J}^*$, ${\mathbb{E}}[\operatorname{dist}(X(T),B)]$ and $m_B(T)$ are estimated by the $N_\text{eval}$ particles.

\subsubsection{Topic-directed search on the PubMed Diabetes citation graph}
\label{sec:pubmed-navigation}

We evaluate our approach on the PubMed Diabetes benchmark subset \cite{sen2008collective,yang2016planetoid}, a citation network where nodes represent diabetes-related publications and edges denote symmetrized citation links. The connected graph comprises $|V|=19{,}717$ nodes and $|E|=44{,}324$ undirected edges, partitioned into three topical classes: Experimental ($4103$ nodes), Type~2 ($7739$ nodes), and Type~1 Diabetes ($7875$ nodes).

Following the graph-distance transport formulation of \cref{sec:real-graph-distance-transport}, we designate the Experimental and Type~2 classes as the source set $A$ and target set $B$, respectively. With unit edge weights ($r_{ij}=1$) and a maximum distance threshold $D_B=9$, the remaining distance $\operatorname{dist}(i,B)$ for any $i \in V$ is evaluated in graph hops. Each particle trajectory represents an idealized literature search, governed by a terminal coefficient $\lambda_G=50$. For computational scalability, the discrete nodes $s$ are embedded into $\mathbb{R}^{32}$. The model is trained over $3000$ epochs using $M=40$ time steps and the time step $\Delta t=0.1$.


\begin{table}[htbp]
 \centering
 \small
 \resizebox{0.98\linewidth}{!}{%
 \begin{tabular}{cccccccccc}
 \toprule
 $|V|$ & $|E|$ & $M$ &
 $\Delta t$ & $m_B(T)$ &
 ${\mathbb{E}}[\operatorname{dist}(X(T),B)]$ & $\widehat{\mathcal K}_{\Delta t}$ &
 $\widehat{\mathcal G}_B$ &
 $\widehat{\mathcal J}^*$ \\
 \midrule
 19,717 & 44,324 & 40 & 0.1 & 0.9555 & 0.1038 & 1.5035 & 0.0115 & 1.3284 \\
 \bottomrule
 \end{tabular}%
 }
 \caption{Evaluation results of the Neural SHP algorithm on the PubMed citation graph. }
 \label{tab:pubmed-shp-navigation}
\end{table}

\Cref{tab:pubmed-shp-navigation} summarizes the results evaluated at the optimal epoch. Over the transport time horizon, the target mass $m_B(T)$ successfully increases to $0.9555$, while ${\mathbb{E}}[\operatorname{dist}(X(T),B)]$ is reduced to $0.1038$ hops.


\begin{figure}[htbp]
 \centering
 \begin{minipage}[b]{0.99\textwidth} 
 {
 \begin{minipage}[b]{0.65\linewidth} 
 \centering
 \includegraphics[width=\linewidth]{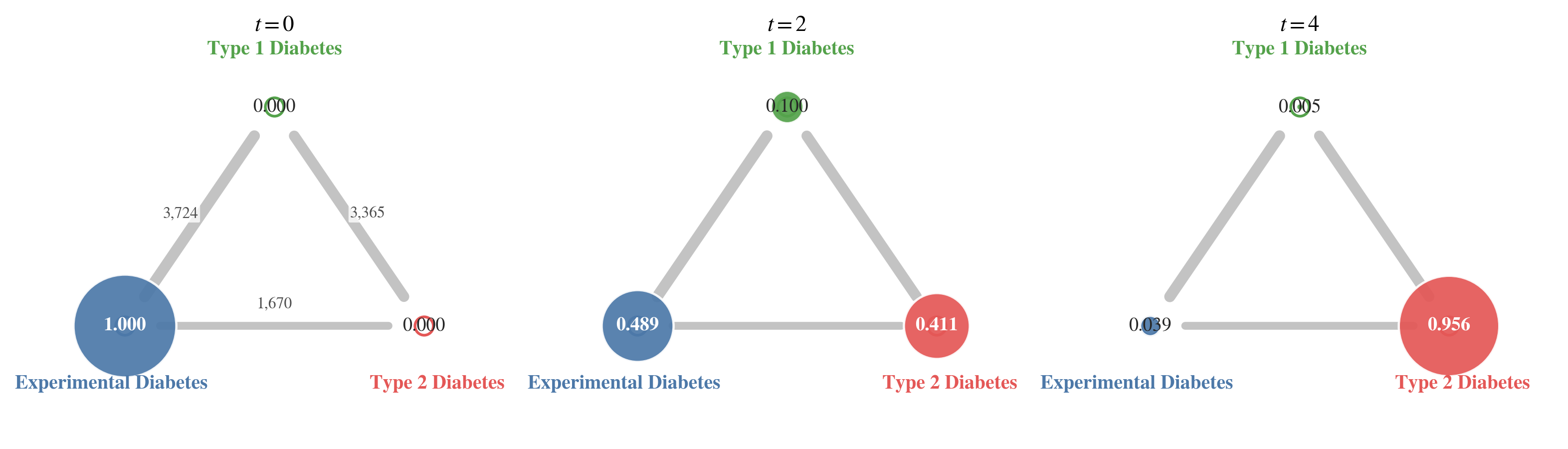}
 \subcaption{Quotient-graph summaries at $t=0,2,4$}\label{subfig:pubmed-quotient-snapshots}
 \end{minipage}
 }
 {
 \begin{minipage}[b]{0.32\linewidth}
 \centering
 \includegraphics[width=\linewidth]{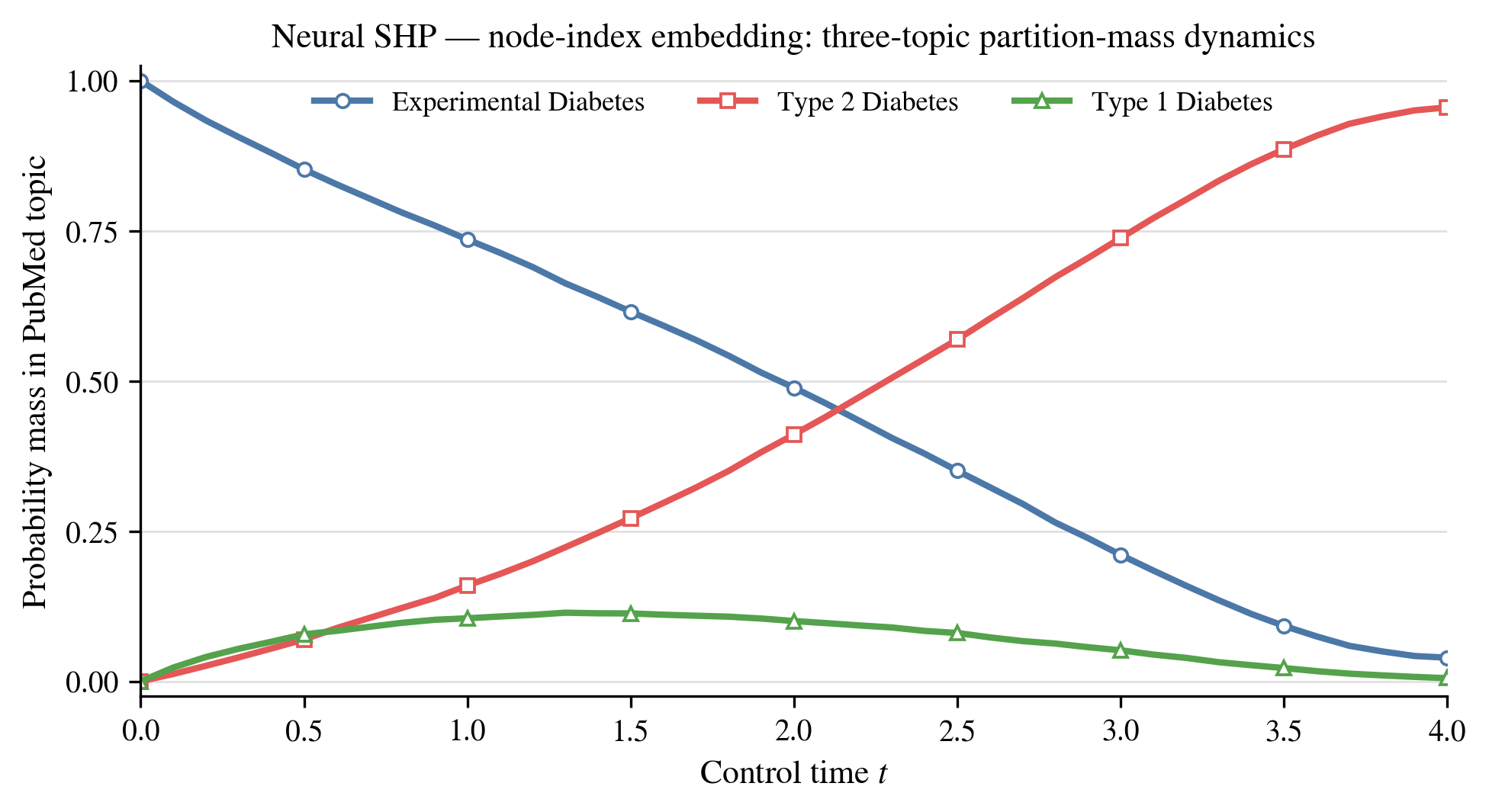}
 \subcaption{Topic dynamics}\label{subfig:pubmed-topic-mass}
 \end{minipage}
 }
 \end{minipage}\caption{Macroscopic transport dynamics on the PubMed citation graph, aggregated over the three fundamental diabetes topics. Panel (a) illustrates snapshots of the quotient-graph flow at times $t=0, 2, 4$, while Panel (b) details the continuous evolution of topic-density distribution over the complete time horizon. }\label{fig:graph-partition-transport-pubmed}
\end{figure}

Figure~\ref{fig:graph-partition-transport-pubmed} summarizes the transport at the topic-level graph partition. Panel~\cref{subfig:pubmed-quotient-snapshots} shows three quotient-graph snapshots,
while panel \cref{subfig:pubmed-topic-mass} shows the corresponding topic-mass dynamics. The particles starting uniformly from the Experimental Diabetes papers are transported to the Type~2 Diabetes papers, with a small fraction of particles remaining in the Type~1 Diabetes papers. The dynamics are computed on the full graph visualized in \cref{fig:pubmed-shp-navigation}, and the quotient graph visualized in \cref{subfig:pubmed-quotient-snapshots} is the partited graph corresponding to the topics. Marker area is proportional to nodal density. 

\begin{figure}[htbp]
 \centering
 \includegraphics[width=\linewidth]{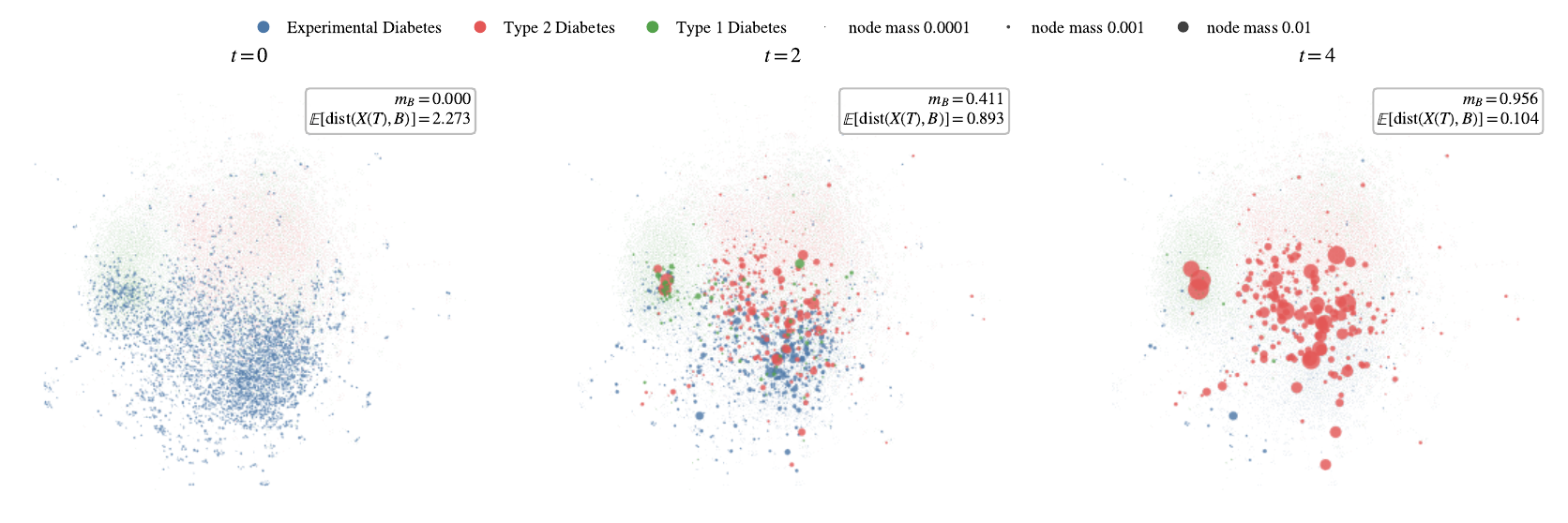}
 \caption{Evolution of Neural SHP on the processed PubMed citation graph from the Experimental Diabetes class to the Type~2 Diabetes class at intermediate control times $t=0, 2, 4$. The node positions are used only for visualization. }
 \label{fig:pubmed-shp-navigation}
\end{figure}

\subsubsection{Road routing from Atlanta to Athens on the graph}
\label{sec:georgia-atlanta-athens-routing}

We evaluate the  origin-destination (OD) task on the undirected Georgia TIGER road graph
\cite{dimacs2006challenge9}. Its largest connected component has
$731{,}954$ nodes and $852{,}705$ undirected edges, average degree $2.3299$, and maximum degree $d_\text{max}=7$. 
The overview in~\cref{subfig:georgia-road-overview} is the graph of Georgia road-network, while particle dynamics is shown on~\cref{subfig:georgia-snapshots}.

Using the notations of \cref{sec:real-graph-distance-transport}, let $a$ and $b$ be the starting node of all the particles at Atlanta and the destination node at Athens. Here $A=\{a\}$, $B=\{b\}$, $r_{ij}$ is the road length in kilometers, and $$
D_B=L^*=\operatorname{dist}(a,B)=107.9694\ \mathrm{km}.
$$ The reference shortest path has $489$ edges. 

For this experiment, let $\tau=t/T\in[0,1]$ denote the normalized time, and let $s_i\in[0,1]^2$ denote the normalized two-dimensional map coordinate of node $i$. 
We set $\lambda_K=0.5$ and $\lambda_G=3000$. The function $G_B(i)=\operatorname{dist}(i,B)/L^\ast$ is the normalized weighted shortest-path distance from node $i$ to the destination $B$. 

The experimental settings are $M=1500$, $\Delta t=0.2$, and learning rate $10^{-4}$. 
The experiments ran for a maximum of $10{,}000$ epochs to illustrate the training dynamics.

\begin{remark}[Terminal-Fixed Variant of the Potential Network]
For the long-range Georgia routing problem, we employ a variant of the potential network, $S^\theta$  incorporating a terminal-fixed structure to enforce the terminal condition exactly. We parameterize the potential as
\begin{equation}
S^\theta(s_i,t)
=
-\tau\frac{\lambda_G}{\lambda_K} G_B(i)
+(1-\tau)S^\theta_R(s_i,\tau).
\end{equation}
Here $S^\theta_R$ is a residual potential network with the same feed-forward architecture as the unconstrained potential network, taking $s_i$ and $\tau$ as inputs. Consequently, the residual vanishes at $t=T$ and
$
S^\theta(s_i,T)=-\frac{\lambda_G}{\lambda_K} G_B(i).
$ This architectural modification is used exclusively in this case to simplify the training process.
\end{remark}

\begin{table}[htbp]
 \centering
 \small
 \resizebox{\linewidth}{!}{%
 \begin{tabular}{lcccccccc}
  \toprule
  $|V|$ & $|E|$ & $M$ & $\Delta t$ & $m_B(T)$ &
  ${\mathbb{E}}[\operatorname{dist}(X(T),B)]$ (km) &
  $\widehat{\mathcal K}_{\Delta t}$ &
  $\widehat{\mathcal G}_B$ &
  $\widehat{\mathcal J}^*$ \\
  \midrule
  $731{,}954$ & $852{,}705$ &1500 &0.2 & 0.8902 & 1.0287 & 3204.8796 & 0.009528& 1631.0231 \\
  \bottomrule
 \end{tabular}%
 }
 \caption{Neural SHP on the Georgia road-network. The terms are measured at the lowest $\hat {\mathcal J}^*$.} \label{tab:georgia-atlanta-athens-routing}
\end{table}

\Cref{tab:georgia-atlanta-athens-routing} summarizes the results. The target mass reach to $m_B(T)=0.8902$ at epoch 10{,}000 and ${\mathbb{E}}[\operatorname{dist}(X(T),B)]=1.0287$ km. The evaluated objective is $\widehat{\mathcal J}^*=1631.0231.$
Before training, the same terminal-conditioned initialization gives $m_B(T)=0.7580$, ${\mathbb{E}}[\operatorname{dist}(X(T),B)]=2.2014$ km, and $\widehat{\mathcal J}=3554.6155$. The terminal functional minimizes expected remaining distance.

\Cref{fig:georgia-atlanta-athens} are the visualization results of the OD task by Neural SHP on the Georgia road-network. \Cref{subfig:georgia-road-overview} shows the Georgia road-network with the reference shortest path from Atlanta to Athens. \Cref{subfig:georgia-snapshots} shows the result snapshots at $t=0,150,300$. The marked area shows empirical nodal density, and color shows clipped road-distance progress $\min\{1,\max\{0,1-\operatorname{dist}(i,B)/L^*\}\}$.

\begin{figure*}[htbp] 
 \centering
 \begin{minipage}[b]{0.99\textwidth} 
 {
 \begin{minipage}[b]{0.23\linewidth} 
 \centering
 \includegraphics[width=\linewidth]{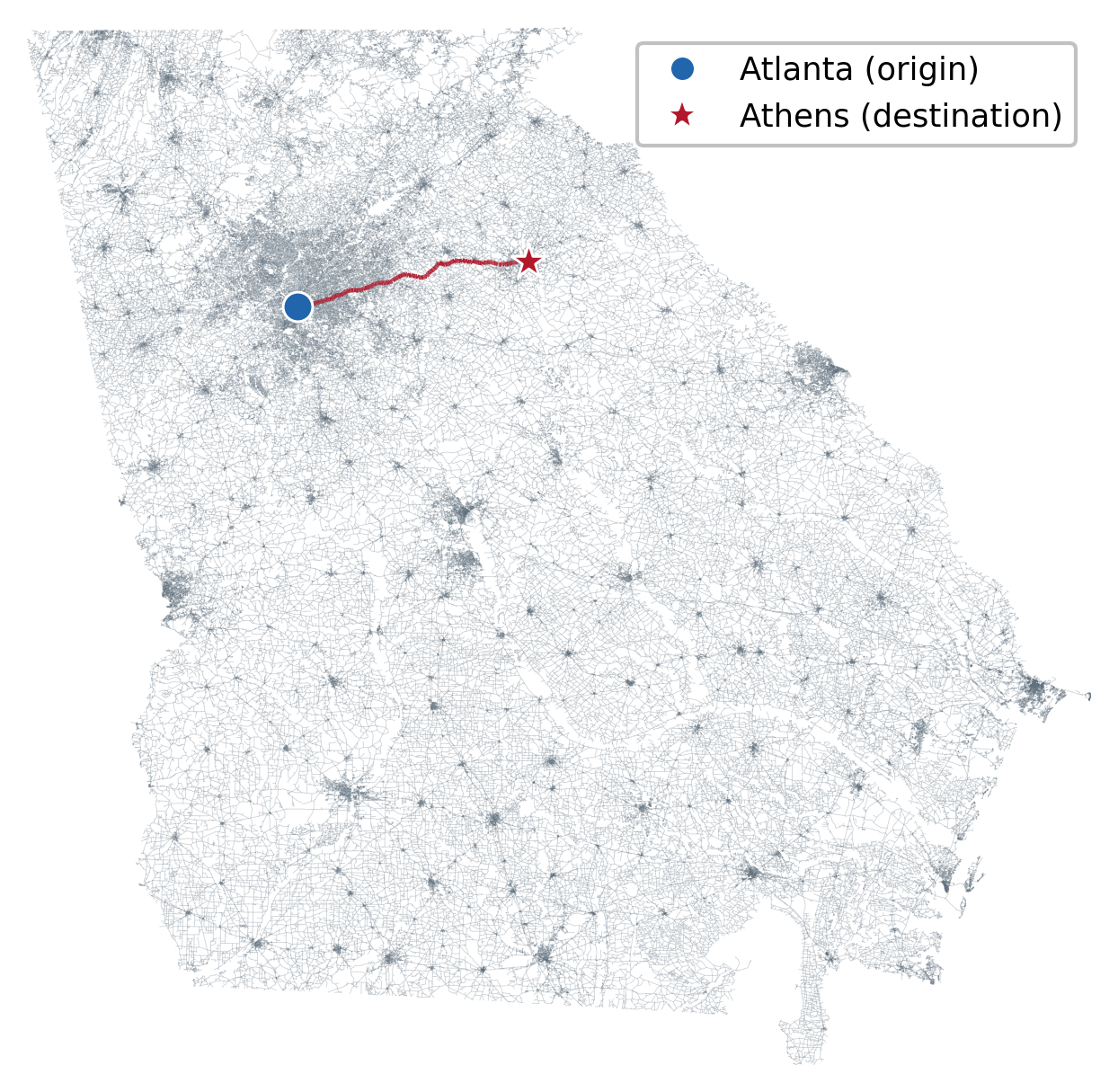}
 \subcaption{Road-network.}\label{subfig:georgia-road-overview}
 \end{minipage}
 }
 {
 \begin{minipage}[b]{0.75\linewidth}
 \centering
 \includegraphics[width=\linewidth]{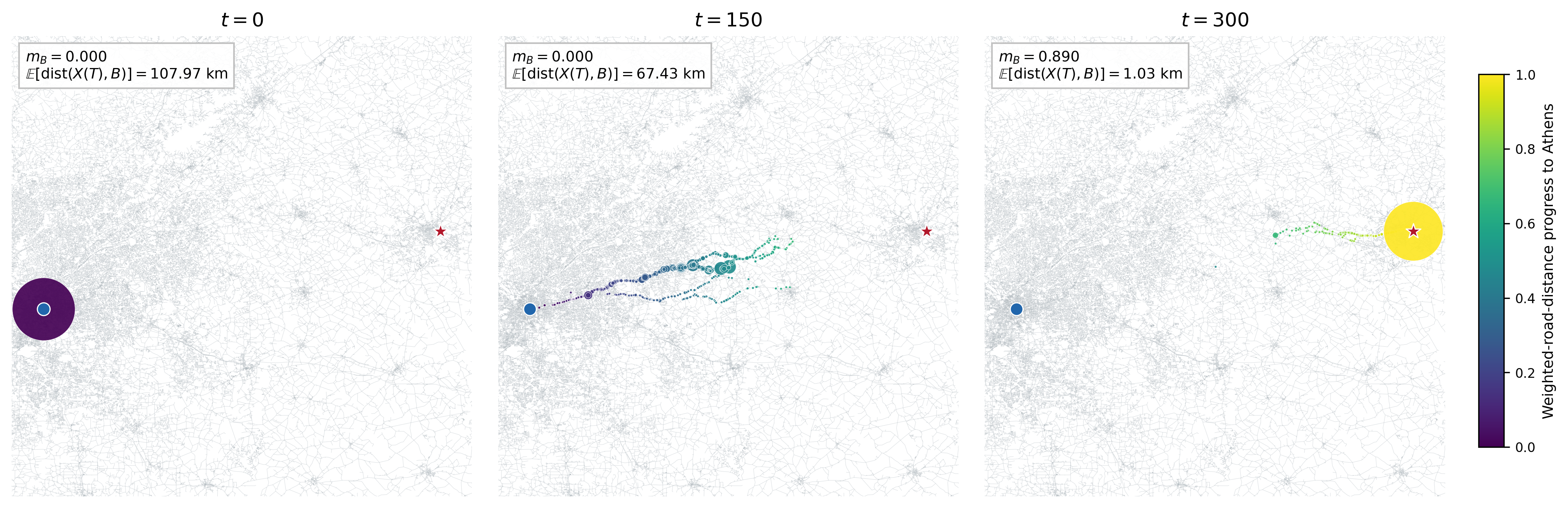}
 \subcaption{Snapshots in $t=0,150,300$}\label{subfig:georgia-snapshots}
 \end{minipage}
 }
 \end{minipage}
 \caption{Panel (a) is the Georgia road-network with the reference shortest path. Panel (b) shows the Neural SHP results on the Georgia road-network at $t=0,150,300$. 
 }\label{fig:georgia-atlanta-athens}
\end{figure*}

Neural WHF is not run on this graph. In the current full-density implementation, differentiating through $M=1500$ steps would require more than 
$562.14$ GB. This exceeds the $16$ GB memory service on which the Neural SHP is tested on.

\section*{Acknowledgments}
This work follows up YXF's previous work during internship under the supervision of Professor Zhou, in the School of Mathematics, Georgia Institute of Technology. YXF and YX was partially supported by the Project of Hetao Shenzhen-HKUST Innovation Cooperation Zone HZQB-KCZYB-2020083.

\bibliographystyle{siamplain}
\bibliography{references}

\end{document}